\documentclass[10pt,a4paper,oneside,reqno]{amsart}

\usepackage{setspace}
\usepackage[totalwidth=14cm,totalheight=21.5cm]{geometry}
\usepackage[T1]{fontenc}
\usepackage[dvips]{graphicx}
\usepackage{epsfig}
\usepackage{amsmath,amssymb,amscd,amsthm}
\usepackage{subfigure}
\usepackage[latin1]{inputenc}
\usepackage{latexsym}
\usepackage{graphics}
\usepackage{colortbl} 
\usepackage{color}
\usepackage{ae}
\usepackage{enumitem}
\usepackage[all]{xy}
\usepackage{scalerel}
\usepackage{tikz}
\usepackage{cancel}
\usepackage{bbm,amsbsy}
\usepackage{mathrsfs}  
\usepackage[colorlinks=true,pagebackref=true]{hyperref}
\usepackage[normalem]{ulem}
\usepackage{nicefrac}
\usepackage{xfrac}
\usepackage{faktor}
\usepackage{tikz-cd} 

\usepackage{nomencl}
\usepackage[normalem]{ulem}

\makenomenclature

\makeatletter
\newtheorem*{rep@theorem}{\rep@title}
\newcommand{\newreptheorem}[2]{%
\newenvironment{rep#1}[1]{%
 \def\rep@title{#2 \ref{##1}}%
 \begin{rep@theorem}}%
 {\end{rep@theorem}}}
\makeatother

\makeatletter
\newtheorem*{rep@cor}{\rep@title}
\newcommand{\newrepcor}[2]{%
\newenvironment{rep#1}[1]{%
 \def\rep@title{#2 \ref{##1}}%
 \begin{rep@cor}}%
 {\end{rep@cor}}}
\makeatother

\makeatletter
\newtheorem*{rep@prop}{\rep@title}
\newcommand{\newrepprop}[2]{%
\newenvironment{rep#1}[1]{%
 \def\rep@title{#2 \ref{##1}}%
 \begin{rep@prop}}%
 {\end{rep@prop}}}
\makeatother

\newtheorem{corollary}{Corollary}[section]

\newtheorem{theorem}[corollary]{Theorem}

\newtheorem{proposition}[corollary]{Proposition}

\newrepcor{corollary}{Corollary}
\newreptheorem{theorem}{Theorem}
\newrepprop{proposition}{Proposition}

\newtheorem{lemma}[corollary]{Lemma}

\theoremstyle{definition}
\newtheorem{definition}[corollary]{Definition}
\theoremstyle{remark}
\newtheorem{remark}[corollary]{Remark}
\newtheorem*{remark*}{Remark}
\newtheorem{example}[corollary]{Example}

\newtheorem*{notation*}{Notation}

\newcommand{\R}{\mathbb{R}}
\newcommand{\HH}{\mathcal{H}}
\newcommand{\Hyp}{\mathbb{H}}

\newcommand{\AdS}{\mathbb{A}\mathrm{d}\mathbb{S}}
\newcommand{\dS}{\mathrm{d}\mathbb{S}}
\newcommand{\isom}{\mathrm{Isom}}
\newcommand{\stab}{\mathrm{Stab}}

\newcommand{\pd}{\partial}

\begin{document}

\setcounter{secnumdepth}{2}
\setcounter{tocdepth}{2}

\title[The half-space model of pseudo-hyperbolic space]{The half-space model of pseudo-hyperbolic space}

\author[Andrea Seppi]{Andrea Seppi}
\address{Andrea Seppi: Institut Fourier, UMR 5582, Laboratoire de Math\'ematiques,
Universit\'e Grenoble Alpes, CS 40700, 38058 Grenoble cedex 9, France..} \email{andrea.seppi@univ-grenoble-alpes.fr}

\author[Enrico Trebeschi]{Enrico Trebeschi}
\address{Enrico Trebeschi: Dipartimento di Matematica,
AMS - Universit\`a di Bologna, Via Zamboni, 33, 40126 Bologna, Italy.}
 \email{enrico.trebeschi@studio.unibo.it}

\thanks{The first author is member of the national research group GNSAGA}

\maketitle

\begin{abstract}
In this note we develop a half-space model for the pseudo-hyperbolic space $\Hyp^{p,q}$, for any $p,q$ with $p\geq 1$. This half-space model embeds isometrically onto the complement of a degenerate totally geodesic hyperplane in $\Hyp^{p,q}$. We describe the geodesics, the totally geodesic submanifolds, the horospheres, the isometry group in the half-space model, and we explain how to interpret the boundary at infinity in this setting.
\end{abstract}

\section{Introduction}
The \emph{pseudo-hyperbolic space} $\Hyp^{p,q}$ of signature $(p,q)$ is the space
$$\Hyp^{p,q}=\{X\in \R^{p,q+1}\,|\,\langle X,X\rangle=-1\}/\{\pm\mathrm{Id}\}~,$$
where $\R^{p,q+1}$ is the pseudo-Euclidean space of signature $(p,q+1)$. It is equipped with a pseudo-Riemannian metric induced by the bilinear form of $\R^{p,q+1}$, which makes it a geodesically complete pseudo-Riemannian manifold of constant sectional curvature $-1$. See for instance \cite{ctt,dgk,one:sem} for more details on $\Hyp^{p,q}$ in arbitrary signature and dimension. When $q=0$, $\Hyp^{n,0}$ is isometric to the \emph{hyperbolic space} $\Hyp^n$; for $q=1$, $\Hyp^{n,1}$ is the \emph{Anti-de Sitter space}, see for instance \cite{bee:ehr} and \cite{bon:sep}. 

The purpose of this note is to describe a \emph{half-space} model for $\Hyp^{p,q}$. This is defined as the open half-space $\{z>0\}$ in $\R^{p+q}$ endowed with the pseudo-Riemannian metric
$$\frac{dx_1^2+\ldots+dx_{p-1}^2-dy_1^2-\ldots-dy_{q}^2+dz^2}{z^2}~,$$
and denoted $\HH^{p,q}$.
When $q=0$, this is the usual half-space model of $\Hyp^n$. When $q\geq 1$, this space is not globally isometric to $\Hyp^{p,q}$ (which is indeed not simply connected, hence not even homeomorphic to the half-space). In Section \ref{sec:defi} we will show that $\HH^{p,q}$ embeds isometrically into $\Hyp^{p,q}$, with image the complement of a totally geodesic degenerate hyperplane. In other words, $\HH^{p,q}$ is \emph{not} a geodesically complete pseudo-Riemannian manifold.

We remark that the half-space model for the Anti-de Sitter space $\Hyp^{2,1}$ has been introduced by Danciger in \cite{dan:phd}, and has been used for instance in \cite{tam:zei}; in any dimension, the half-space model of $\Hyp^{n,1}$ also appears in \cite{rio:sep}. Of course one analogously defines the pseudo-spherical space $\mathbb S^{p,q}$, by taking the space of vectors $X$ such that $\langle X,X\rangle=1$ in $\R^{p+1,q}$, and taking the quotient by $\{\pm\mathrm{Id}\}$. The space obtained in this way is anti-isometric to $\Hyp^{q,p}$. A half-space model of $\mathbb S^{p,q}$ is defined similarly, provided $q\geq 1$. As a particular case, the half-space model of the de Sitter space $\mathbb S^{p,1}$ has been studied in \cite{nom:hal}. We decided to focus on the case of $\Hyp^{p,q}$ for the sake of definiteness: up to changing a sign to the pseudo-Riemannian metric, one recovers the half-space model for $\mathbb S^{q,p}$ if $p\geq 1$.

The main results of this note aim at describing the geometry of $\HH^{p,q}$. We give a classification result for the totally geodesic subspaces of any dimension (Section \ref{sec:tot}), and then a more refined classification of the geodesics (Section \ref{sec:geo}). We will see that spacelike geodesics (\emph{i.e.} those for which the tangent vector is positive for the pseudo-Riemannian metric) are of four types: vertical lines, ellipses, half-parabolas and half-branches of hyperbolas, all orthogonal to the horizontal hyperplane (\emph{i.e.} the boundary of the half-space). The first two types include the usual geodesics in the half-space model of the hyperbolic space. Lightlike and timelike geodesics (\emph{i.e.} those for which the tangent vector is respectively null and negative) are lines and branches of hyperbolas respectively. In Section \ref{sec:bdy} we provide a description of the boundary at infinity $\pd_\infty\Hyp^{p,q}$, seen from the half-space model. Of course the boundary $\pd\HH^{p,q}$ in $\R^{p+q}$, which is a copy of the pseudo-Euclidean space $\R^{p-1,q}$, is conformally embedded in $\pd_\infty\Hyp^{p,q}$; we describe topologically its compactification $\pd_\infty\HH^{p,q}$ in terms of divergence of totally geodesic degenerate hypersurfaces. In Section \ref{sec:horo} we describe the horospheres in the half-space model. Finally, in Section \ref{sec:isometry}, we compute the isometry group $\isom(\HH^{p,q})$, as a result of the analysis of geodesics. We remark that this does not correspond to the isometry group of $\Hyp^{p,q}$, but only to a subgroup that preserves the complement of a degenerate hyperplane. Nevertheless, we study the action of $\isom(\Hyp^{p,q})$ on the half-space model in terms of $\isom(\HH^{p,q})$ and some transformations which are the analogue of inversions in hyperbolic geometry.

\subsection*{Acknowledgements} We would like to thank Stefano Francaviglia for his interest and encouragement. The first author is grateful to J\'er\'emy Toulisse for related discussions. We would like to thank an anonymous referee for useful comments.

\section{First definitions and properties}\label{sec:defi}

We start by recalling the standard definition of pseudo-hyperbolic space. Given integers $p,q\geq 0$, we define:

$$\widetilde\Hyp^{p,q}=\{X\in \R^{p,q+1}\,|\,\langle X,X\rangle=-1\}~,$$
where $\R^{a,b}$ is the pseudo-Euclidean space of signature $(a,b)$ and $\langle\cdot,\cdot\rangle$ denotes the corresponding bilinear form, namely if $X=(X_1,\ldots,X_{a+b})$ and $Y=(Y_1,\ldots,Y_{a+b})$ then
$$\langle X,Y\rangle=X_1Y_1+\dots X_aY_a-X_{a+1}Y_{a+1}-\ldots-X_{a+b}Y_{a+b}~.$$

 It is known that $\widetilde\Hyp^{p,q}$, endowed with the restriction of the bilinear form $\langle\cdot,\cdot\rangle$, is a pseudo-Riemannian manifold of signature $(p,q)$ of constant sectional curvature $-1$. Then we define the pseudo-hyperbolic space as
$$\Hyp^{p,q}=\widetilde\Hyp^{p,q}/\{\pm\mathrm{Id}\}~.$$
Since $\pm\mathrm{Id}$ acts isometrically on $\widetilde\Hyp^{p,q}$, $\Hyp^{p,q}$ inherits a pseudo-Riemannian metric of constant sectional curvature $-1$.

Recall also that the boundary at infinity of pseudo-hyperbolic space is defined as:
$$\partial_\infty\Hyp^{p,q}=\mathrm{P}\{X\in \R^{p,q+1}\,|\,\langle X,X\rangle=0\}\subset\R\mathrm P^{p+q}~,$$
namely the projectivised cone of null (\emph{i.e.} isotropic) vectors.

\subsection{The half-space model}

Let us now introduce the half-space model of $\Hyp^{p,q}$, which is the main object of this paper. In the following, we will always assume $p\geq 1$.

\begin{definition}
The \emph{half-space model} of pseudo-hyperbolic geometry is the space 
$$\HH^{p,q}=\{(x,y,z)\in \R^{p-1}\oplus\R^q\oplus\R\,|\,z>0\}~,$$
endowed with the pseudo-Riemannian metric
\begin{equation}\label{eq:metric}
g_{p,q}=\frac{dx_1^2+\ldots+dx_{p-1}^2-dy_1^2-\ldots-dy_{q}^2+dz^2}{z^2}~.
\end{equation}
\end{definition}

We will also denote by $\partial\HH^{p,q}$ the boundary of $\HH^{p,q}$ in $\R^{p+q}$, namely the hyperplane $\{z=0\}$. We mention here some well-known specific cases.

\begin{example}
When $q=0$, $\HH^{n,0}$ is the usual half-space model of the hyperbolic space $\Hyp^n=\Hyp^{n,0}$. 
\end{example}

\begin{example}
When $q=1$, $\HH^{n-1,1}$ is the so-called half-space model of the Anti-de Sitter space $\AdS^n=\Hyp^{{n-1,1}}$. This has been introduced for $\AdS^3$ in \cite{dan:phd}, and has been applied for instance in \cite{tam:zei}, and in \cite{rio:sep} in arbitrary dimension.
\end{example}

\begin{example}
The pseudo-hyperbolic space $\Hyp^{1,n-1}$ coincides with the $n$-dimensional de Sitter space $\dS^{n-1,1}$ up to changing the sign to the metric tensor. The half-space model in this case, again up to changing the sign, appeared in \cite{nom:hal}.
\end{example}

\subsection{An isometric embedding}

It can be checked directly that the sectional curvature of $\HH^{p,q}$ is constant and equal to $-1$, for any $p,q$. Nevertheless, this follows from the next proposition, which justifies the claim that $\HH^{p,q}$ is a ``model'' for the pseudo-hyperbolic space $\Hyp^{p,q}$.

\begin{proposition}\label{emb}
There exists an isometric embedding 
$$\iota_{p,q}\colon\HH^{p,q}\to \Hyp^{p,q}~.$$
If $q=0$, $\iota_{p,q}$ is surjective. Otherwise, its image is the complement of a totally geodesic degenerate hyperplane in $\Hyp^{p,q}$.
\end{proposition}
\begin{proof}
We will first define an embedding $\tilde\iota_{p,q}\colon\HH^{p,q}\to \widetilde\Hyp^{p,q}\subset\R^{p,q+1}$. 
To simplify the notation, we will write $$h(x,y)=x_{1}^{2}+\dots+x_{p-1}^{2}-y_{1}^{2}-\dots-y_{q}^{2}$$
for $(x,y)\in\R^{p-1}\oplus\R^q$. Then define $\tilde\iota_{p,q}(x,y,z)=(X_1,\ldots,X_{p+q+1})$ where:

\begin{align*}
		&X_{i}=\frac{x_{i}}{z} && i=1,\dots p-1,\\
		&X_{p}=\frac{1-h(x,y)-z^{2}}{2z} &&\\	
		&X_{j+p}=\frac{y_{j}}{z} && j=1,\dots q,\\
		&X_{p+q+1}=\frac{1+h(x,y)+z^{2}}{2z}. &&
\end{align*}
One checks immediately that $X_{p}^{2}-X_{p+q+1}^{2}=-1-{h(x,y)}/{z^{2}}$, hence $\tilde{\iota}_{p,q}$ takes values in $\widetilde\Hyp^{p,q}$, namely $\langle\tilde{\iota}_{p,q}(x,y,z),\tilde{\iota}_{p,q}(x,y,z)\rangle=-1$.

To prove that $\tilde{\iota}_{p,q}$ is an isometry, one can easily compute the differential:
\begin{align*}
	d\tilde{\iota}_{p,q}\left(\frac{\pd}{\pd x_{i}}\right)=&\frac{1}{z}\frac{\pd}{\pd X_{i}}-\frac{x_{i}}{z}\frac{\pd}{\pd X_{p}}+\frac{x_{i}}{z}\frac{\pd}{\pd X_{p+q+1}}\\
	d\tilde{\iota}_{p,q}\left(\frac{\pd}{\pd y_{j}}\right)=&\frac{y_{j}}{z}\frac{\pd}{\pd X_{p}}+\frac{1}{z}\frac{\pd}{\pd X_{j+p}}-\frac{y_{j}}{z}\frac{\pd}{\pd X_{p+q+1}}\\
	d\tilde{\iota}_{p,q}\left(\frac{\pd}{\pd z}\right)=&-\sum_{i=1}^{p-1}\frac{x_{i}}{z^{2}}\frac{\pd}{\pd X_{i}}-\frac{1-h(x,y)+z^{2}}{2z^{2}}\frac{\pd}{\pd X_{p}}-\\
	&-\sum_{j=1}^{q}\frac{y_{j}}{z^{2}}\frac{\pd}{\pd X_{j+p}}-\frac{1+h(x,y)-z^{2}}{2z^{2}}\frac{\pd}{\pd X_{p+q+1}}
\end{align*}
and then an easy calculation shows that $\tilde{\iota}_{p,q}^{*}\langle\cdot,\cdot\rangle$ equals the metric tensor \eqref{eq:metric}.

Let us now show that 
\begin{equation}\label{eq:aussois}
\tilde{\iota}_{p,q}(\HH^{p,q})=\widetilde\Hyp^{p,q}\cap\{X_{p}+X_{p+q+1}>0\}~.
\end{equation}
The inclusion $\subseteq$ is trivial as $X_{p}+X_{p+q+1}={1}/{z}>0$. For the other inclusion, given $(X_1,\ldots,X_{p+q+1})$ such that $X_1^2+\ldots+X_p^2-X_{p+1}^2-\ldots X_{p+q+1}^2=-1$ and $X_{p}+X_{p+q+1}>0$, define
\begin{align*}
		&x_{i}=\frac{X_{i}}{X_{p}+X_{p+q+1}} && i=1,\dots p-1\\
		&y_{j}=\frac{X_{j+p}}{X_{p}+X_{p+q+1}} && j=1,\dots q\\
		&z=\frac{1}{X_{p}+X_{p+q+1}}~.
\end{align*}
Then one checks that $\tilde{\iota}_{p,q}(x,y,z)=(X_1,\ldots,X_{p+q+1})$.

Incidentally, in the above argument we constructed an inverse of $\tilde\iota_{p,q}$ over its image, which implies that $\tilde\iota_{p,q}$ is injective. It also follows from \eqref{eq:aussois} that the restriction of $\pi$ to the image of $\tilde{\iota}_{p,q}$ is injective, where $\pi$ is the projection from $\widetilde \Hyp^{p,q}$ to $\Hyp^{p,q}$. Hence,
defining $\iota_{p,q}=\pi\circ\tilde{\iota}_{p,q}$, $\iota_{p,q}$ is an isometric embedding whose image is the complement of $P\cap\Hyp^{p,q}$, where $P$ is the hyperplane defined by the condition $X_{p}+X_{p+q+1}=0$. It is known that (when $q\geq 1$) this is a totally geodesic hyperplane in $\Hyp^{p,q}$, which is degenerate because $P$ is degenerate in $\R^{p,q+1}$, being the orthogonal complement of the line spanned by the isotropic vector $\pd/\pd X_p-\pd/\pd X_{p+q+1}$.
Observe that for $q=0$, the intersection $P\cap\Hyp^{p,0}$ is empty, hence we recover that $\iota_{p,0}$ is a global isometry between the half-space model and the hyperboloid model of the hyperbolic space.
\end{proof}

\subsection{Symmetries}\label{sym}

The following lemma, which also serves as a definition for the group $G$, introduces some symmetries of the model $\HH^{p,q}$. 

\begin{lemma}\label{lemma:G}
The group
$$G=\{(x,y,z)\mapsto \lambda(A(x,y)+(x_0,y_0),z)\,|\,\lambda>0,A\in\mathrm{O}(p-1,q),(x_0,y_0)\in \R^{p-1}\oplus\R^q\}$$
is a subgroup of the isometry group $\isom(\HH^{p,q})$.
Moreover, $G$ acts transitively on $\HH^{p,q}$ and the stabilizer of a point is isomorphic to $\mathrm{O}(p-1,q)$.
\end{lemma}
\begin{proof}
It is immediate to check that $G$ is a subgroup. To see that it acts by isometries on $\HH^{p,q}$, it suffices to show that the transformations of the form
\begin{align*}
(x,y,z)&\mapsto \lambda(x,y,z)\\
(x,y,z)&\mapsto (A(x,y),z)\\
(x,y,z)&\mapsto (x+x_0,y+y_0,z)
\end{align*}
preserve the metric tensor \eqref{eq:metric}, which we leave as an easy exercise.

The action of $G$ is transitive since the map $$(x,y,z)\mapsto z_{0}\left(x+\frac{x_0}{z_0},y+\frac{y_0}{z_0},z\right)$$ sends $(0,0,1)$ to $(x_{0},y_{0},z_{0})$. Then an isometry of $G$ fixes $(0,0,1)$ if and only if $\lambda=1$, $x_{0}=y_{0}=0$, \emph{i.e.} $\stab_{G}(0,0,1)\cong\mathrm{O}(p-1,q)$.
\end{proof}

We will see in Theorem \ref{cor:G} that $G$ is actually the full isometry group $\isom(\HH^{p,q})$ when $q\geq 1$. Since every local isometry between open neighbourhoods of $\Hyp^{p,q}$ extends to a global isometry, the isometric embedding $\iota_{p,q}$ induces a group monomorphism from $G$ to $\isom(\Hyp^{p,q})$, which is clearly not surjective because in $\isom(\Hyp^{p,q})$ there are isometries that do not preserve the totally geodesic hyperplane whose complement is the image of $\iota_{p,q}$. (Indeed if $n=p+q$, then $G$ is a Lie group of dimension $(n^2-n+2)/2$, while $\isom(\Hyp^{p,q})$, which is isomorphic to a double quotient of $\mathrm{O}(p,q+1)$, has dimension $n(n+1)/2=\dim G+n-1$).

\section{Totally geodesic submanifolds}\label{sec:tot}
The next step in our analysis is the study of the totally geodesic subspaces of $\HH^{p,q}$. In the following, when referring to a totally geodesic submanifold (and in the particular case of geodesics and totally geodesic hypersurfaces) we will always implicitely assume that they are \emph{maximal}, \emph{i.e.} not properly included in any other totally geodesic submanifold of the same dimension.

\subsection{The geodesic equations}

We start by writing the geodesic equations. One can easily check that the only nonvanishing Christoffel symbols of $\HH^{p,q}$ are
	\begin{align*}
		&\Gamma_{x_{i},z}^{x_{i}}=\Gamma_{z,x_{i}}^{x_{i}}=-1/z && i=1,\dots p-1,\\
		&\Gamma_{y_{j},z}^{y_{j}}=\Gamma_{z,y_{j}}^{y_{j}}=-1/z && j=1,\dots q,\\
		&\Gamma_{x_{i},x_{i}}^{z}=1/z && i=1,\dots p-1,\\
		&\Gamma_{y_{j},y_{j}}^{z}=-1/z && j=1,\dots q,\\
		&\Gamma_{z,z}^{z}=-1/z.
	\end{align*}	
	In the following, by a small abuse of notation, given a point $(x,y,z)\in\HH^{p,q}$ (with $x\in \R^{p-1}$, $y\in\R^q$, $z\in\R$), we will denote by $\|\cdot\|$ the Euclidean norm of $x$ and $y$, namely:
	
	$$\|x\|^2=\sum_{i=1}^{p-1}x_i^2\qquad \|y\|^2=\sum_{j=1}^{q}y_j^2~.$$

	From the above computation of Christoffel symbols, we obtain that a curve $\gamma(t)=(x(t),y(t),z(t))$ is geodesic if and only if the following system of ODEs is satisfied:
	\begin{align}\label{geo}
		\begin{cases}
			x''-\frac{2z'}{z}x'=0\\
			y''-\frac{2z'}{z}y'=0\\
			z''+\frac{1}{z}\left(\|x'\|^{2}-\|y'\|^{2}-|z'|^{2}\right)=0
		\end{cases}.
	\end{align}
As a consequence of this expression of the geodesic equations, we show here that vertical subspaces of any dimension are totally geodesic. For non-degenerate subspaces, this could be proved also by a standard symmetry argument, by finding an isometry whose fixed points set coincides with the subspace. But the argument below works for degenerate subspaces as well. 

\begin{proposition}\label{vert}
	Every submanifold of the form 
	$$V_\ell:=\{(x,y,z)\in\R^{p-1}\oplus\R^{q}\oplus\R\,|\,(x,y)\in \ell, z>0\}~,$$ for $\ell$ an affine subspace of $\R^{p+q-1}$, is totally geodesic.
\end{proposition}
\begin{proof}
Let us define $\ell$ as the set of solutions of a finite number of affine conditions of the form
\begin{equation}\label{eq:affine subspace}
\sum_{i=1}^{p-1} a_ix_i+\sum_{j=1}^{q} b_jy_j=c~.
\end{equation}
We claim that if $\gamma$ is a geodesic such that $\gamma'(0)$ is tangent to the subspace $V_\ell$, namely 
\begin{equation}\label{eq:initial condition}
\sum_{i=1}^{p-1} a_ix'_i(0)+\sum_{j=1}^{q} b_jy'_j(0)=0~,
\end{equation}
then $\gamma(t)$ satisfies \eqref{eq:affine subspace} for all times of definition. This clearly implies that $V_\ell$ is totally geodesic.

To show the claim, define the function 
$$\chi(t)=\sum_{i=1}^{p-1} a_ix_i'(t)+\sum_{j=1}^{q} b_jy_j'(t)~.$$
Taking a linear combination of the equations \eqref{geo}, $\chi$ satisfies the following ODE: $\chi'(t)=f(t)\chi(t)$, for $f(t)=-2z'(t)/z(t)$. By our hypothesis \eqref{eq:initial condition}, $\chi(0)=0$, hence $\chi\equiv 0$. This implies that $\gamma(t)$ identically satisfies \eqref{eq:affine subspace} and concludes the proof.
\end{proof}

\subsection{Totally geodesic hypersurfaces}

We now give the classification of totally geodesic submanifolds of codimension one. The general case (\emph{i.e.} in any codimension) will then follow in Theorem \ref{thm:tot_geo_submanifolds}.

\begin{proposition}\label{prop:tot_geo_hyper}
	The totally geodesic hypersurfaces of $\HH^{p,q}$ are precisely:
	\begin{enumerate}
		\item the vertical hyperplanes $V_{\mathcal{L}}$, for $\mathcal{L}$ an affine hyperplane in $\pd\HH^{p,q}$;
		\item the quadric hypersurfaces of the form 
		\begin{equation}\label{eq:quadric}
			\|x-x_0\|^2-\|y-y_0\|^2+z^{2}=c,\quad  c\in\R \tag{Q}
		\end{equation} 
	\end{enumerate}
	for some $(x_0,y_0)\in\pd\HH^{p,q}$. The hypersurfaces of the former type are degenerate if and only if $\mathcal{L}$ is degenerate in $\R^{p-1,q}$, and have signature $(m+1,n)$ where $(m,n)$ is the signature of $\mathcal{L}$. Those of the latter type are degenerate if and only if $c=0$, and have signature $(p-1,q-1)$ if $c=0$, $(p,q-1)$ if $c<0$, and  $(p-1,q)$ if $c>0$.
	\end{proposition}

See also Figure \ref{fig:quadrics} for some pictures in dimension 3.

\begin{figure}[htb]
\centering
\begin{minipage}[c]{.55\textwidth}
\centering
\includegraphics[height=5.1cm]{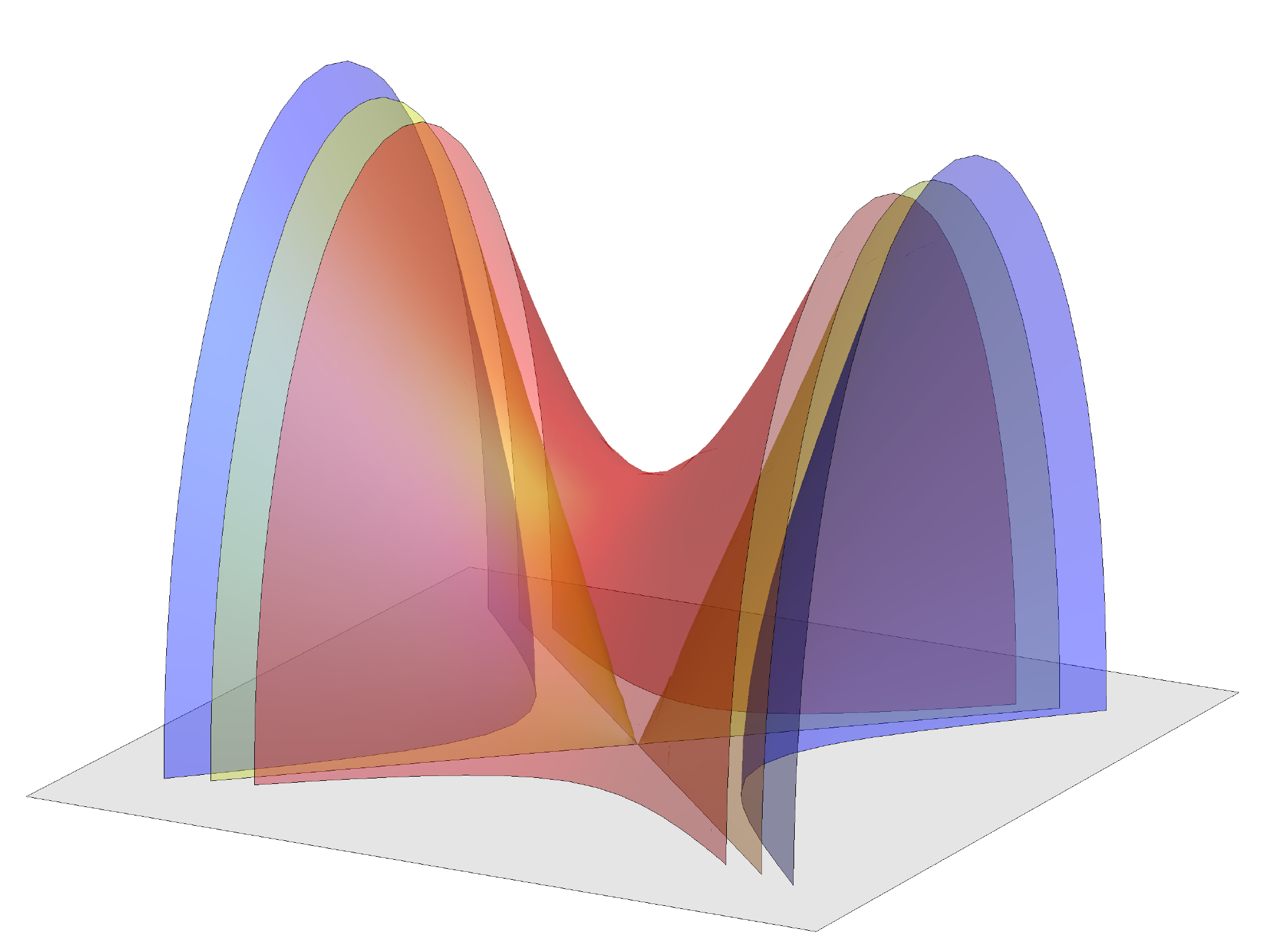}
\end{minipage}%
\begin{minipage}[c]{.45\textwidth}
\centering
\includegraphics[height=5.1cm]{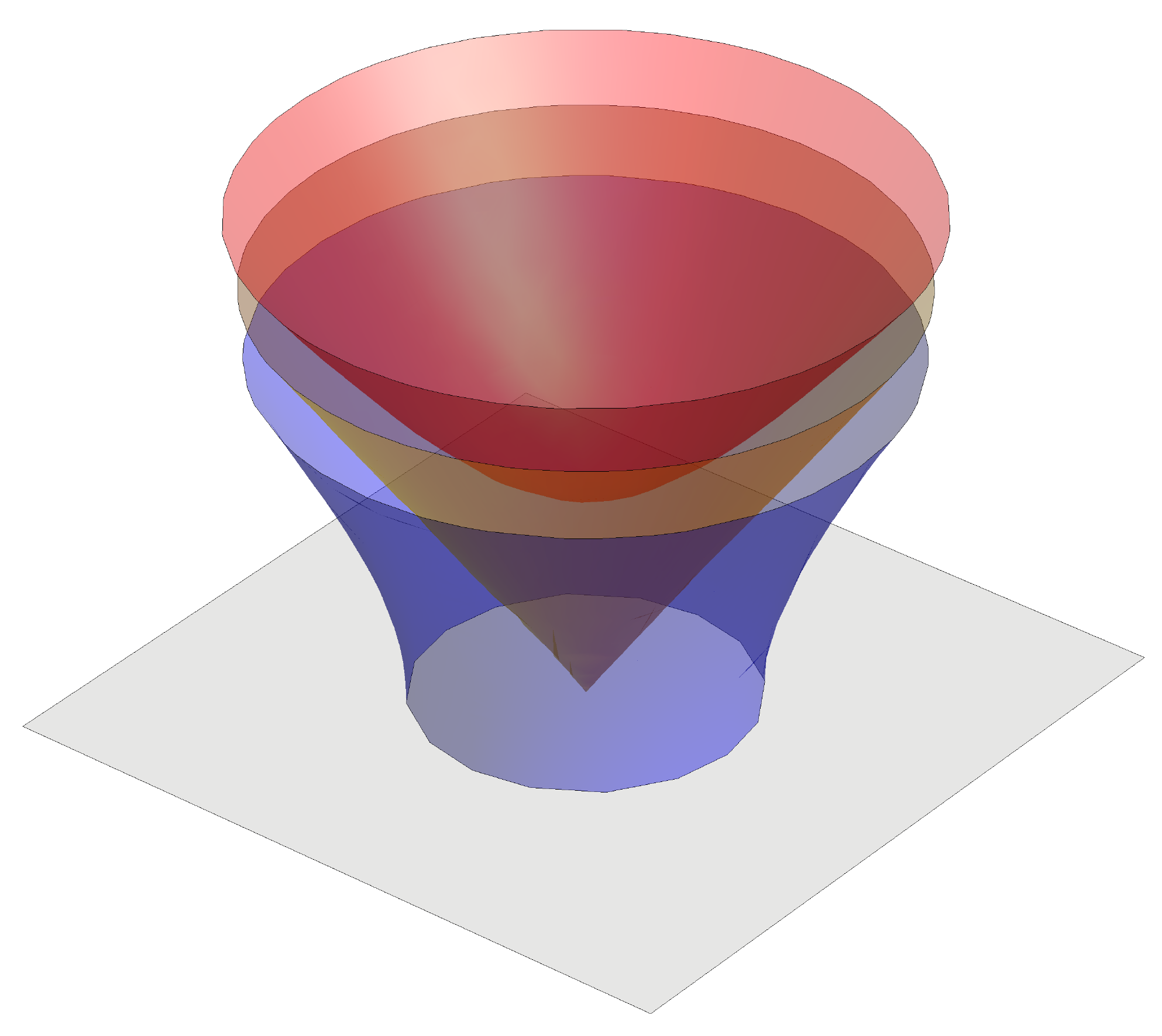}
\end{minipage}%
\caption{The totally geodesic quadric hypersurfaces in $\HH^{2,1}$ (left) and $\HH^{1,2}$ (right).}\label{fig:quadrics}
\end{figure}

\begin{proof}
	It has been proved in Proposition \ref{vert} that vertical hyperplanes are totally geodesic. To prove that the quadric hypersurfaces as in the statement are totally geodesic, we will show that the intersection of the quadric hypersurface with any vertical 2-plane $V_\ell$ (for $\ell$ a line) is a geodesic of $\HH^{p,q}$. This clearly implies that the hypersurface is totally geodesic, since any ambient geodesic that is tangent to the hypersurface at time zero remains in the hypersurface for all times.
	
	To see this, we parameterize such a curve as  
	$$\gamma(t)=(u_{0}+tu,v_0+tv,f(t)),$$
	where the function $f$ is determined by the quadric Equation \eqref{eq:quadric}, namely
	$$f(t)=\sqrt{c-\|x(t)-x_0\|^2+\|y(t)-y_0\|^2}~,$$
	which is defined for $t$ in some interval $I$. 

	As in \eqref{geo} one can compute 	
	\begin{equation}\label{eq:cov dev}
		\nabla_{\gamma'}\gamma'=\left(-\frac{2f'}{f}u,-\frac{2f'}{f}v,f''+\frac{\|u\|^{2}-\|v\|^{2}-|f'|^{2}}{f}\right).
	\end{equation}
Remarking that $f(t)^{2}=c-\|tu+u_0-x_0\|^{2}+\|tv+v_0-y_0\|^{2}$, differentiating twice one obtains
$$(f')^{2}+f''f=-\|u\|^{2}+\|v\|^{2}.$$
Then the last term in \eqref{eq:cov dev} becomes $-({2f'}/{f})f'$, that means $\nabla_{\gamma'}\gamma'=-({2f'}/{f})\gamma'$, \emph{i.e.} $\gamma$ in an unparameterized geodesic.

It only remains to show that these are all totally geodesic hypersurfaces. For this purpose, we choose any vector $(u,v,w)$ tangent to $\HH^{p,q}$ at a point $(x,y,z)$ and we show that there exists a totally geodesic hypersurface (which is necessarily unique) of the above two forms containing the point $(x,y,z)$ and whose tangent space is orthogonal to $(u,v,w)$. The key observation is that the orthogonality can be computed with respect to the flat metric 
$$\overline g_{p,q}=dx_1^2+\ldots+dx_{p-1}^2-dy_1^2-\ldots-dy_{q}^2+dz^2~,$$ namely by seeing the half-space as a subset of a pseudo-Euclidean space of signature $(p,q)$, since $\overline g_{p,q}$ is conformal to the metric $g_{p,q}$.

If $w=0$, then clearly the hypersurface we are looking for is $V_\mathcal{L}$, for $\mathcal{L}$ the affine subspace of $\R^{p-1}\oplus\R^q$ containing the point $(x,y)$ and whose underlying vector space is the orthogonal of $(u,v)$. So let us now assume $w\neq 0$. Let $(x_0,y_0,0)$ be the point of intersection of the line through $(x,y,z)$ having direction $(u,v,w)$ with $\pd\HH^{p,q}$. The quadric hypersurfaces of the form \eqref{eq:quadric} are precisely the sets of points at constant squared distance from $(x_0,y_0,0)$ for the conformal pseudo-Euclidean metric $\overline g_{p,q}$. Hence there is one quadric hypersurface that contains the point $(x,y,z)$, and it is orthogonal to the position vector $(x,y,z)-(x_0,y_0,0)$, which is proportional to $(u,v,w)$ by construction. The statement about the signature is easily checked, by using again the conformal pseudo-Euclidean metric $\overline g_{p,q}$, since the signature of a submanifold only depends on the conformal class of the metric. This concludes the proof.
\end{proof}

\begin{remark}\label{rem:lightlike by embedding}
For degenerate hypersurfaces, a direct computation, very similar to the proof of Theorem \ref{prop:horo} (but setting the constant $a=0$), shows that vertical hyperplanes and quadric hypersurfaces of the form \eqref{eq:quadric} with $c=0$ are the preimages under $\iota_{p,q}$ of totally geodesic degenerate hypersurfaces in $\Hyp^{p,q}$. Indeed, the latter are formed by the double quotient of the set the vectors $X\in\widetilde\Hyp^{p,q}$ such that $\langle X,V\rangle=0$, for some null vector $V$. If we pick $V$ of the form $(u,0,v,0)$ for $u\in\R^{p-1}$, $v\in\R^q$ and $\|u\|^2-\|v\|^2=0$, a direct computation shows that the preimage is 
	$|x\cdot u-y\cdot v|=0$. If instead we pick instead $V=(0,1,0,1)$, the preimage is $\|x\|^2-\|y\|^2+z^2=0$. Up to translation, this concludes the claim of the remark. See Theorem \ref{prop:horo} for more details.
\end{remark}

\subsection{The general classification}

We can finally state the classification result for totally geodesic submanifolds.

\begin{theorem}\label{thm:tot_geo_submanifolds}
	The totally geodesic submanifolds of $\HH^{p,q}$ are precisely:
	\begin{enumerate}
		\item the vertical subspaces,
		\item the intersections of quadric hypersurfaces of the form 
		\eqref{eq:quadric} with a vertical subspace.
	\end{enumerate}
\end{theorem}
\begin{proof}
The submanifolds in the statement are totally geodesic: for the first item this follows from Proposition \ref{vert}, while for the second item from Proposition \ref{prop:tot_geo_hyper} and the fact that the intersection of totally geodesic submanifolds is totally geodesic. To show that they exhaust the totally geodesic submanifolds, pick any linear subspace $W$ of $T_{(x,y,z)}\HH^{p,q}$. We will show that there exists a (necessarily unique) submanifold of the above two forms tangent to $W$ at $(x,y,z)$.  If $W$ contains the vertical direction, then we can write $W=W_0\oplus\partial/\partial z$, for $W_0$ the orthogonal complement of $\partial/\partial z$ in $W$. Denoting by $\ell$ the affine subspace through the point $(x,y)$ with underlying vector space $W_0$, $V_\ell$ is a totally geodesic subspace tangent to $W$ at $(x,y,z)$.

Now suppose that $W$ does not contain $\partial/\partial z$, and extend $W$ to a subspace $W_1$ of codimension one which is still transverse to the vertical direction.  By the proof of Proposition \ref{prop:tot_geo_hyper}, there exists a quadric hypersurface $Q$ which is tangent to $W_1$ at $(x,y,z)$. Also, as in the first part of this proof, we find a vertical subspace $V_\ell$ which is tangent to $W\oplus \partial/\partial z$. Then $Q\cap V_\ell$ is tangent to $W$ at $(x,y,z)$. This concludes the proof.
\end{proof}

\section{Geodesics}\label{sec:geo}
The next step in our analysis is the study of the geodesics of $\HH^{p,q}$. We will divide our analysis in three cases, namely lightlike, timelike and spacelike geodesics.

We recall that a parameterized geodesic is \emph{complete} if it is defined over $\R$. Similarly, a geodesic is complete \emph{on one side} if it admits a parameterization defined on a half-open interval $[a,+\infty)$. In the spacelike and timelike case, this corresponds to requiring that, for any parameterization, the integral of the square root of the absolute value of the quadratic form applied to the tangent vector is infinite over the corresponding end.

To simplify the statements of the following propositions, we refer to geodesics as unparametrized, \emph{i.e.} our statements are actually about the image of the parametrized curves.

\subsection{Lightlike geodesics}

We start by lightlike geodesics, namely those for which the tangent vector is isotropic for the metric tensor \eqref{eq:metric}.

\begin{proposition}\label{lightlike}
	Lightlike geodesics in $\HH^{p,q}$ are precisely the straight lines spanned by a lightlike vector. These are incomplete as they escape from compact sets of $\HH^{p,q}\cup\pd\HH^{p,q}$, unless they are contained in a horizontal hyperplane $\{z=c\}$.
\end{proposition}
\begin{proof}
Let $(u,v,w)$ be a lightlike vector, that is, $\|u\|^2-\|v\|^2+|w|^2=0$. Up to changing the sign, we can assume $w\ge 0$. We claim that if $w> 0$, then $$\gamma(t)=(x_0,y_0,0)+(1/t)(u,v,w)$$ is a parameterized geodesic; if instead $w=0$, then $$\gamma(t)=(x_0,y_0,z_0)+t(u,v,0)$$ is a parameterized geodesic. This clearly implies the statement: these are all the lightlike geodesics because, up to choosing the parameter $t$ suitably, one finds such a geodesic with tangent vector a multiple of $(u,v,w)$ through any point of $\HH^{p,q}$. Moreover, geodesics of the former type are defined on $(0,\infty)$, hence they are complete only when they approach $\pd\HH^{p,q}$, while those of the latter type are defined on $\R$.

The claim is an easy computation from Equations \eqref{geo}. Indeed, since $\gamma'$ is lightlike, we have $\|x'\|^{2}-\|y'\|^{2}=-|z'|^{2}$, hence the equations become 
\begin{equation}\label{eq:pippo}
(x'',y'',z'')=(2z'/z)(x',y',z')~,
\end{equation} and one immediately checks that both expressions above for $\gamma$ satisfy \eqref{eq:pippo}.
\end{proof}

See Figure \ref{fig:lightcone} to visualize the cone of lightlike geodesics emanating from a point in $\HH^{2,1}$ and $\HH^{1,2}$.

\begin{figure}[htb]
\centering
\begin{minipage}[c]{.5\textwidth}
\centering
\includegraphics[height=4.8cm]{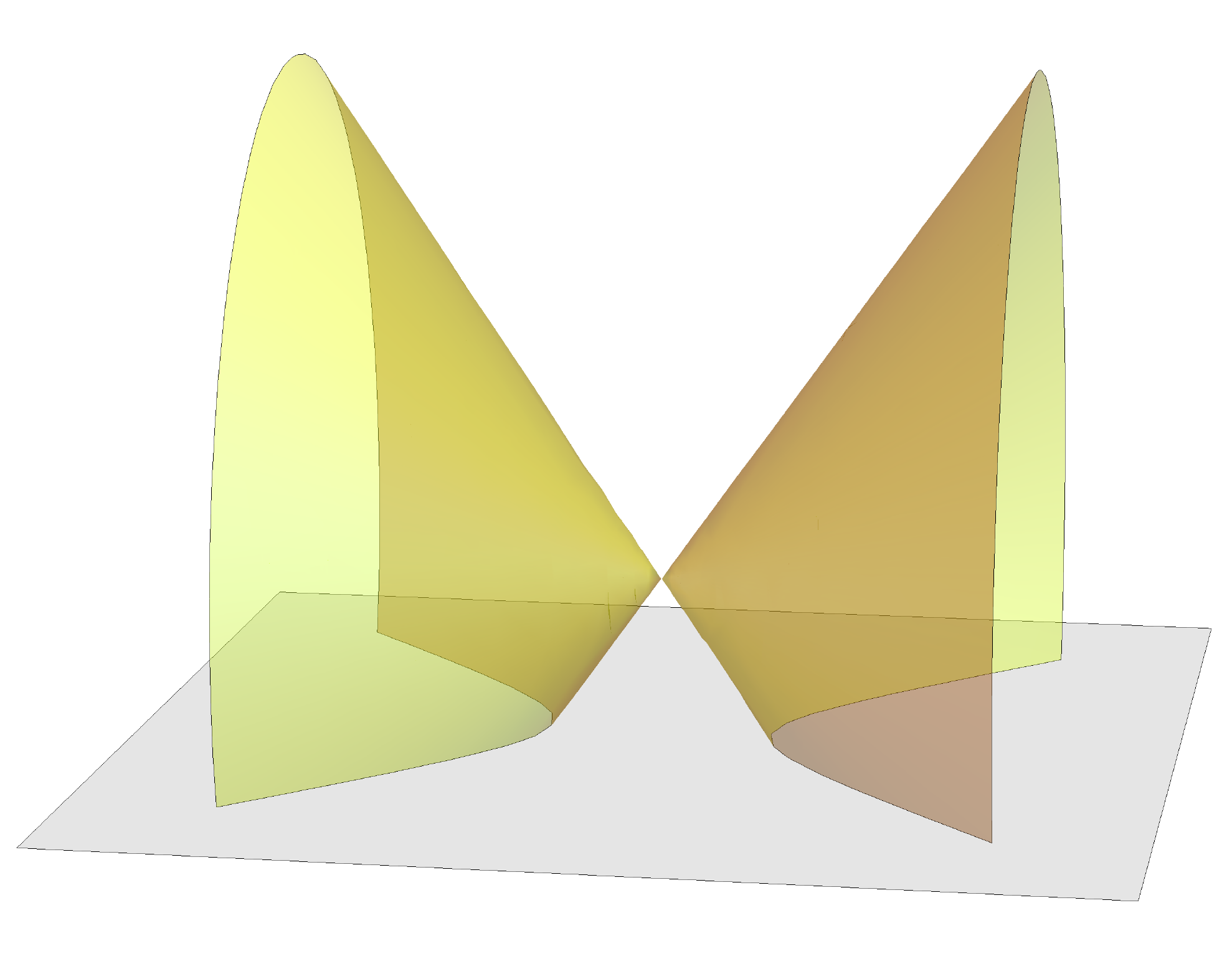}
\end{minipage}%
\begin{minipage}[c]{.5\textwidth}
\centering
\includegraphics[height=4.8cm]{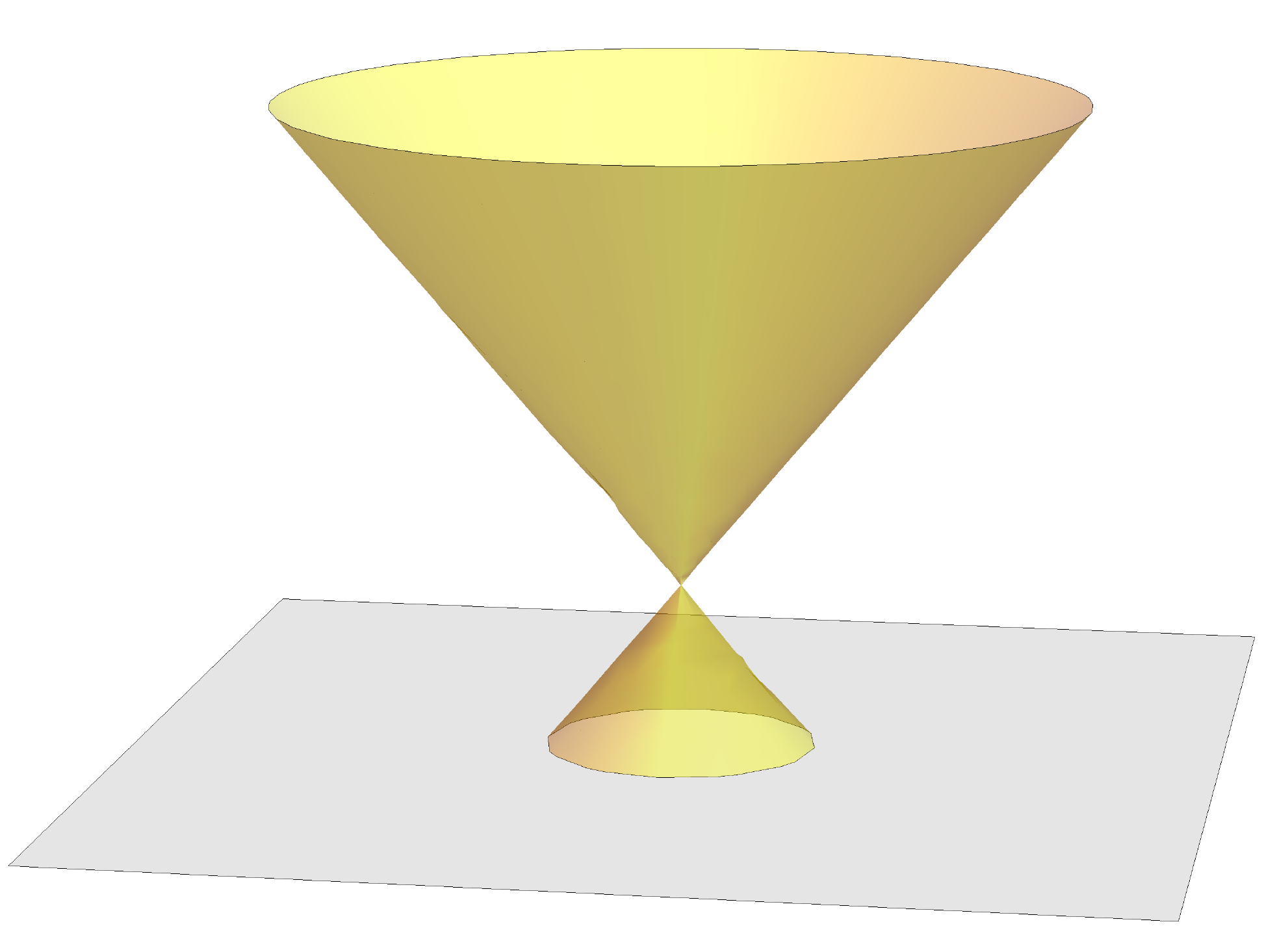}
\end{minipage}%
\caption{The lightcone from a point in $\HH^{2,1}$ (left) and $\HH^{1,2}$ (right).}\label{fig:lightcone}
\end{figure}

\begin{remark}
The fact that unparameterized lightlike geodesics are straight lines can also be proved by observing that $\HH^{p,q}$ is conformal to the upper half-space in $\R^{p,q}$ endowed with the restriction of the pseudo-Euclidean metric, and applying the fact that two conformal pseudo-Riemannian metrics have the same unparameterized lightlike geodesics (see \cite[Proposition 2.131]{ghlf}).
\end{remark}

\subsection{A preliminary computation}
As a consequence of the classification in Theorem \ref{thm:tot_geo_submanifolds}, geodesics are either straight vertical lines or conics. We will give here a more precise classification in terms of the  eccentricity, computed with respect to the Euclidean distances in $\HH^{p,q}\subset\R^{p+q}$. We start by a general computation that we will apply repeatedly.

\begin{lemma}\label{lemma:prel comp}
	Geodesics of $\HH^{p,q}$ are precisely:
	\begin{enumerate}
		\item vertical lines;
		\item conics of equation
		\begin{equation}\label{eq:conic0}
			\frac{\|u\|^{2}-\|v\|^{2}}{\|u\|^{2}+\|v\|^{2}}s^{2}+z^{2}+2As=C\,,\qquad A,C\in\R \tag{Q'}
		\end{equation} 
		with respect to Euclidean coordinates $(s,z)$ on a vertical 2-plane $V_\ell$, where the underlying vector space of $\ell$ is spanned by $(u,v)$.
	\end{enumerate}
\end{lemma}
\begin{proof}
By Theorem \ref{thm:tot_geo_submanifolds}, geodesics are either
 vertical lines or obtained 
	 intersecting \eqref{eq:quadric} with the a 2-plane $V_\ell$. Up to a horizontal translation, we can assume that the line $\ell$ contains the origin, hence it can be  parameterized as
\begin{equation}\label{eq:xys}
			(x,y)=\frac{s}{\sqrt{\|u\|^{2}+\|v\|^{2}}}(u,v)~.
\end{equation}
A direct computation from \eqref{eq:quadric} yields the Equation \eqref{eq:conic0}
 with the constants $A=-(u\cdot x_0-v\cdot y_0)/\sqrt{\|u\|^{2}+\|v\|^{2}}$ and $C=c-\|x_0\|^{2}+\|y_0\|^{2}$.
 \end{proof}
 Observe that if $\|u\|^{2}-\|v\|^{2}\neq 0$, then replacing the coordinate $s$ by $s-s_0$ for a suitable choice of $s_0$, and relabeling the constant $C$, one then obtains the following equation: 
\begin{equation}\label{eq:conic}
			\frac{\|u\|^{2}-\|v\|^{2}}{\|u\|^{2}+\|v\|^{2}}s^{2}+z^{2}=C\,,\qquad C\in\R~. \tag{Q''}
		\end{equation} 

\begin{remark}\label{rmk:lightlike as intersections}
Lemma \ref{lemma:prel comp} provides yet another method to obtain the straight lines as lightlike geodesics. Consider Equation  \eqref{eq:conic} for $C=0$, namely $$\frac{\|u\|^{2}-\|v\|^{2}}{\|u\|^{2}+\|v\|^{2}}s^{2}+z^{2}=0$$
which is indeed a double line through the origin if $\|u\|^{2}-\|v\|^{2}<0$. An immediate computation shows that these lines are indeed lightlike. Similarly, considering \eqref{eq:conic0} with $C>0$, $A=0$ and $\|u\|^{2}-\|v\|^{2}=0$, we obtain a horizontal line with lightlike direction. A similar approach will be used in the next sections for the analysis of spacelike and timelike geodesics.
\end{remark}

\subsection{Timelike geodesics}

We now move on to the study of timelike geodesics. Let us remark that if a vector $(u,v,w)$ is tangent to a timelike curve at any point, then necessarily $\|u\|<\|v\|$ by the expression of the metric \eqref{eq:metric}. 

\begin{proposition}\label{timelike}
	Timelike geodesics in $\HH^{p,q}$ are exactly the branches of hyperbola with center on $\pd\HH^{p,q}$, which do not meet $\pd\HH^{p,q}$, of eccentricity 
	$$e_T(u,v)=\sqrt{1+\frac{\|u\|^2+\|v\|^2}{\|v\|^2-\|u\|^2}}~,$$ 
	where $(u,v,w)$ is a vector tangent to the geodesic at any point. These are incomplete on both sides.
\end{proposition}

See also Figure \ref{fig:geodesics}. 
\begin{proof}
    We will consider timelike geodesics as the intersections of quadric hypersurfaces with a vertical plane, as in Equation \eqref{eq:conic0}. In order to get a timelike geodesic, the vertical plane is necessarily of signature $(1,1)$, hence  $\|u\|<\|v\|$, in which case we can reduce to Equation \eqref{eq:conic}. As we observed in Remark \ref{rmk:lightlike as intersections}, if $C=0$ then Equation \eqref{eq:conic} gives a pair of lightlike lines with the same endpoint on $\partial\HH^{p,q}$.
	If $C<0$, we obtain a pair of hyperbolas meeting $\pd\HH^{p,q}$ orthogonally. Since the half-space metric is conformal to the pseudo-Euclidean metric on $\R^{p+q}$, these are spacelike (they tend to be vertical as they approach $\pd\HH^{p,q}$). 
	
	We are left with the case of $C>0$, which gives indeed a hyperbola that does not meet $\pd\HH^{p,q}$. These are easily seen to be timelike, since the tangent vector at the minimum point of the $z$-coordinate along the hyperbola is proportional to $(u,v,0)$, which is timelike by our initial assumption $\|u\|<\|v\|$. The eccentricity is $e_T(u,v)$.
	
	It only remains to show that these are incomplete. First, observe that if $H$ and $H'$ are two such hyperbolas (considered as a subset of $\HH^{p,q}$), then there is an element of $G$ that mapping $H$ to $H'$. Indeed, one can use a translation and a dilatation to map the minimum of the $z$-coordinate on $H$ to that on $H'$. Composing with an isometry of the form $(x,y,z)\mapsto(A(x,y),z)$, one can then map the tangent vector to the tangent vector, and this concludes the claim.
	
	To show incompleteness, it thus suffices to consider the hyperbola $\gamma$ parameterized by 
	$	y_{1}(t)=\sinh(t)$, $z(t)=\cosh(t)$, and all other coordinates identically zero. 
Its length is
 $$L(\gamma)=\int_{-\infty}^{+\infty}\sqrt{|g_{p,q}(\gamma'(t),\gamma'(t))|}dt=\int_{-\infty}^{+\infty}\frac{1}{\cosh(t)}dt=\pi~.$$
	Then all timelike geodesics are incomplete on both sides.
\end{proof}

\begin{figure}
\begin{center}
 \includegraphics[height=5cm]{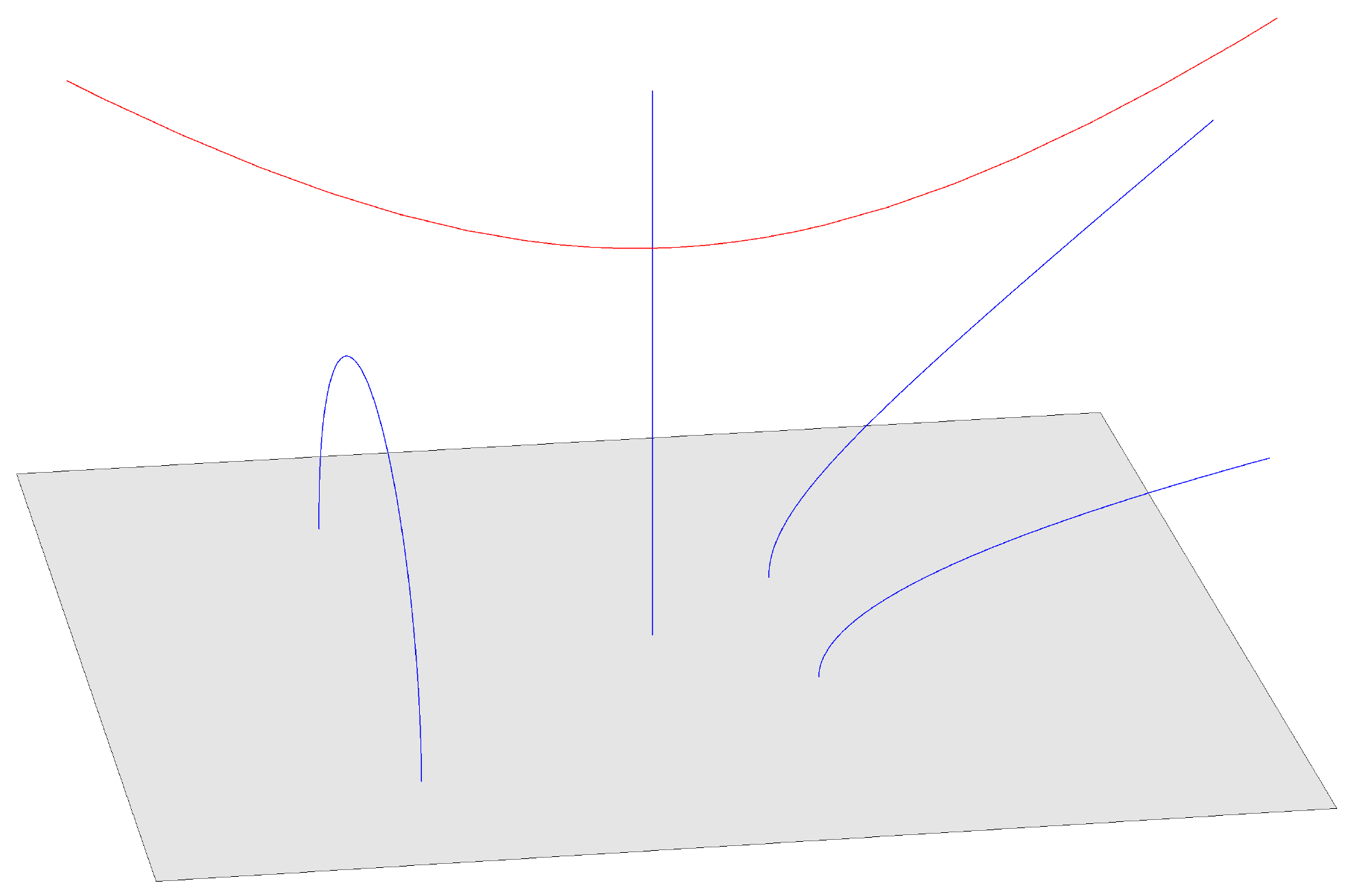}
\end{center}
\caption{A timelike geodesic (in red) and the four types of spacelike geodesics (blue).}\label{fig:geodesics}
\end{figure}

\subsection{Spacelike geodesics}

Finally, we conclude by the analysis of spacelike geodesics. See again Figure \ref{fig:geodesics}.

\begin{proposition}\label{prop:spacelike geo1}
Spacelike geodesics in $\HH^{p,q}$ are exactly of one of the following types:
\begin{enumerate}
\item A vertical straight line;
\item A half-ellipse with foci on $\pd\HH^{p,q}$, of eccentricity $e_S(u,v)$ for $\|u\|>\|v\|$;
\item A parabola with vertex and focus on $\pd\HH^{p,q}$, for $\|u\|=\|v\|$;
\item Half of a branch of hyperbola with foci on $\pd\HH^{p,q}$, meeting $\pd\HH^{p,q}$, of eccentricity $e_S(u,v)$, for $\|u\|<\|v\|$;
\end{enumerate}
where
$$e_S(u,v)=\sqrt{1+\frac{\|v\|^2-\|u\|^2}{\|u\|^2+\|v\|^2}}$$ 
and $(u,v,w)$ is a vector tangent to the geodesic at any point. The first three types are complete, while the fourth type is incomplete as it escapes from compact sets of $\HH^{p,q}\cup\pd\HH^{p,q}$.

\end{proposition}

Before the proof, we observe that in particular all spacelike geodesics meet $\pd\HH^{p,q}$ at right angles with respect to the conformal metric $$\overline g_{p,q}=dx_1^2+\ldots+dx_{p-1}^2-dy_1^2-\ldots-dy_{q}^2+dz^2,$$ which extends over the horizontal hyperplane.
\begin{proof}

		The first type follows from the first point of Lemma \ref{lemma:prel comp}. The arc-length parameterization is $\gamma(t)=(x_0,y_0,e^t)$, which is defined for all times. Let us now consider the second point in Lemma \ref{lemma:prel comp}, by distinguishing three cases according to the sign of $\|u\|^2-\|v\|^2$.
		
		If $\|u\|<\|v\|$, we have already seen in the proof of Proposition \ref{timelike} that Equation \eqref{eq:conic} gives a spacelike geodesic if and only if $C<0$. From the equation, in this case the geodesic is a branch of hyperbola that meets $\partial\HH^{p,q}$ orthogonally, of eccentricity $e_S(u,v)$. To show that it has infinite length when approaching $\pd\HH^{p,q}$ and finite length at the other end, we can assume that $(u,v)=(0,\pd/\pd y_1)$, up to isometry. Up to a translation, we can assume that the curve is parameterized by $y_1(t)=\cosh(t)$, $z(t)=\sinh(t)$, and all the other coordinates are identically zero. A direct computation shows that its length is 
		 $$L(\gamma|[t_0,t_1])=\int_{t_0}^{t_1}\sqrt{g(\gamma'(t),\gamma'(t))}dt=\int_{t_0}^{t_1}\frac{1}{\sinh(t)}dt=\left[\log(\tanh(t/2)\right]_{t_0}^{t_1}~.$$
		 Therefore it is complete as $t_0\to 0^+$, and incomplete as $t_1\to+\infty$.

		If $\|u\|>\|v\|$, then (\ref{eq:conic}) is the equation of an ellipse (for $C>0$) that meets $\partial\HH^{p,q}$ orthogonally, with eccentricity $e_S(u,v)$. We remark that it lies in a positive definite vertical 2-plane, which is isometric to $\HH^{2,0}\cong\Hyp^{2}$, and so is complete.
		
		Finally, when $\|u\|=\|v\|$, \eqref{eq:conic0} becomes $z^{2}+2As=C$.
		We remark that necessarily $A\neq 0$, because otherwise we would obtain a lightlike line as observed in Remark \ref{rmk:lightlike as intersections}. Namely, we obtained a parabola with vertex and focus both on $\pd\HH^{p,q}$.
		To see that it is complete, observe that the last equation of \eqref{eq:cov dev} becomes the ODE $z''=(z')^2/z$, whose solution is $z(t)=z_0 e^{wt/z_0}$, then it is defined for all times. The parameterized geodesic is obtained by setting $s(t)=(1/2A)(C-z^2)$; then $(x(t),y(t))$ is expressed from $s(t)$ as a function of $t$ by Equation \eqref{eq:xys}. This shows that $\gamma$ is complete. 
\end{proof}

As a consequence of this analysis of geodesics, we now have all the tools to prove that the group $G$ introduced in Lemma \ref{lemma:G} is actually the full isometry group of the half-space model $\HH^{p,q}$. However, we postpone the proof to Section \ref{sec:isometry} (see Theorem \ref{cor:G}), where isometries are discussed in greater detail.

\section{The boundary at infinity}\label{sec:bdy}
In this section we will study the boundary at infinity of the pseudo-hyperbolic space
$\HH^{p,q}$ in the half-space model. We first show that the embedding $\iota_{p,q}$, introduced in Proposition \ref{emb}, extends to a non-surjective embedding of $\pd\HH^{p,q}$ into $\pd_\infty\Hyp^{p,q}$; we then describe the missing points and the topology of the boundary in terms of lightlike cones and hyperplanes.

\subsection{The extended embedding}

Recall that Proposition \ref{emb} provided an isometric embedding $\iota_{p,q}$ of $\HH^{p,q}$ into $\Hyp^{p,q}$.

\begin{proposition}
	The isometric embedding $\iota_{p,q}\colon\HH^{p,q}\hookrightarrow\Hyp^{p,q}$ defined in Proposition \ref{emb} extends to an embedding $\pd\HH^{p,q}\hookrightarrow\pd_\infty{\Hyp^{p,q}}$.
\end{proposition}
\begin{proof}
	Consider the embedding $\iota_{p,q}$ as a map from $\HH^{p,q}$ to the projective space $\R\mathrm P^{p+q}$. From Proposition \ref{emb}, it has the expression
		\begin{align*}
		\iota_{p,q}(x,y,z)&=\left[\frac{x}{z}:\frac{1-h(x,y)-z^2}{2z}:\frac{y}{z}:\frac{1+h(x,y)+z^2}{2z}\right]=\\
		&=\left[x:\frac{1-h(x,y)-z^2}{2}:y:\frac{1+h(x,y)+z^2}{2}\right],
	\end{align*}
hence it extends to $\pd\HH^{p,q}=\{z=0\}$ by the above formula.

One can easily check that $\langle\iota_{p,q}(x,y,0),\iota_{p,q}(x,y,0)\rangle=0$, \emph{i.e.} $\iota_{p,q}(\pd\mathcal{H}^{p,q})$ is contained in $\pd_\infty\mathbb{H}^{p,q}$. In particular $\iota_{p,q}(\mathcal{H}^{p,q})\cap\iota_{p,q}(\pd\mathcal{H}^{p,q})=\emptyset$. To show that 
$\iota_{p,q}$ is injective, it therefore suffices to show that it is injective when restricted to $\pd\mathcal{H}^{p,q}$, since we already showed in Proposition \ref{emb} the injectivity of $\iota_{p,q}$ on $\HH^{p,q}$.

For this purpose, suppose $(x,y,0),(t,w,0)\in\pd\mathcal{H}^{p,q}$ are such that $\iota_{p,q}(x,y,0)=\iota_{p,q}(t,w,0)$. It follows from the expression above that $(t,w)=\lambda (x,y)$ for some $\lambda\ne0$ and
\begin{align*}
	\begin{cases}
		\lambda\left(1-h(x,y)\right)=1-h(t,w)=1-\lambda^{2}h(x,y)\\
		\lambda\left(1+h(x,y)\right)=1+h(t,w)=1+\lambda^{2}h(x,y)
	\end{cases},
\end{align*}	which can be rewritten equivalently as
\begin{align*}
	\begin{cases}
		h(x,y)\lambda(1-\lambda)=\lambda-1\\
		h(x,y)\lambda(1-\lambda)=1-\lambda
	\end{cases},
\end{align*}
whose only solution is $\lambda=1$. This concludes that $(t,w,0)=(x,y,0)$.

Moreover, using the same notation as in Proposition \ref{emb}, $X_{p}+X_{p+q+1}=1$, one immediately checks that $\iota_{p,q}(x,y,0)\in\mathrm P\{X_{p}+X_{p+q+1}\ne0\}$. In fact
$$\iota_{p,q}(\pd\HH^{p,q})=\partial_\infty\Hyp^{p,q}\cap\mathrm P\{X_{p}+X_{p+q+1}\ne0\}\subset\R\mathrm P^{p+q}~,$$
because, given a null vector $(X_1,\ldots,X_{p+q+1})$, up to rescaling we can assume $X_{p}+X_{p+q+1}=1$, and a simple algebraic manipulation shows that, for $x_i=X_i$ and $y_j=X_{j+p}$, ${\iota}_{p,q}(x,y,0)=[X_1,\ldots,X_{p+q+1}]$.
\end{proof}

\subsection{The full boundary in the half-space model}\label{sec:boundary hyperplanes}

Our next goal is to describe the entire boundary $\partial_\infty\Hyp^{p,q}$, seen in the half-space model. The starting observation is that $\partial_\infty\Hyp^{p,q}$ is in bijection with the space of degenerate totally geodesic hyperplanes in $\Hyp^{p,q}$. Indeed, to any $X\in \R^{p,q+1}$ such that $\langle X,X\rangle =0$, one associates the intersection of the orthogonal subspace of $X$ with $\Hyp^{p,q}$, more precisely $(X^\perp\cap\widetilde\Hyp^{p,q})/\{\pm\mathrm{Id}\}$, which is a totally geodesic hyperplane in $\Hyp^{p,q}$ of degenerate type. We will simply denote it with $X^\perp$, by a small abuse of notation. 

\begin{example}
The hyperplane ``at infinity'' which is the complement of the embedding $\iota_{p,q}$ in Proposition \ref{emb} is defined by the equation $X_p+X_{p+q+1}=0$, hence it is the orthogonal of any nonzero vector proportional to $\pd/\pd X_p-\pd/\pd X_{p+q+1}$ in $\R^{p,q+1}$.
\end{example}

Clearly two hyperplanes $X^\perp$ and $Y^\perp$ coincide if and only if $X$ and $Y$ are proportional, and every degenerate totally geodesic hyperplane is obtained in this way. It is also clear that the topology of $\pd_\infty\Hyp^{p,q}$ is homeomorphic, under this correspondence, to the Hausdorff topology on closed subsets of $\partial_\infty\Hyp^{p,q}$. (Indeed, a sequence of vectors $X_n$ converges projectively to $X$ if and only if the orthogonal subspace of $X_n$ converges to the orthogonal subspace of $X$.) Motivated by this observation, we give the following definition.

\begin{definition}\label{defi:topology bdy}
We define
$$\pd_\infty\HH^{p,q}=\{\text{degenerate totally geodesic hypersurfaces in }\HH^{p,q}\}\cup\{\infty\}~,$$
where we endow the space of  degenerate totally geodesic hyperplanes with the topology induced by the Hausdorff pseudometric on the closed half-space, and  $\pd_\infty\HH^{p,q}$ with its one-point compactification.
\end{definition}

Recall from Proposition \ref{prop:tot_geo_hyper} that degenerate totally geodesic hypersurfaces in $\HH^{p,q}$ are of the following two types:
\begin{itemize}
\item Vertical hyperplanes $V_{\mathcal{L}}$, where $\mathcal{L}$ is a degenerate affine hyperplane in $\R^{p-1,q}$;
\item Quadric hypersurfaces $Q_{(x_0,y_0)}$ of equation 
	\begin{equation}\label{eq:conex0y0}
			\|x-x_0\|^{2}-\|y-y_0\|^{2}+z^{2}=0
			\end{equation}
\end{itemize}
See also Remark \ref{rem:lightlike by embedding}. The discussion above contains all the elements to prove the following statement, for which we leave the details to the reader. 

\begin{lemma}
The embedding $\iota_{p,q}$ induces a homeomorphism between $\partial_\infty\HH^{p,q}$ and $\partial_\infty\Hyp^{p,q}$, which maps $Q_{(x_0,y_0)}$ to $\iota_{p,q}(x_0,y_0,0)$ and $\infty$ to the projective class of $\pd/\pd X_p-\pd/\pd X_{p+q+1}$ in $\R\mathrm{P}^{p+q}$.
\end{lemma}

\begin{remark}\label{rmk:convergence}
Let us describe more concretely how a sequence of elements in $\pd\HH^{p,q}$, seen as a subset of $\pd_\infty\HH^{p,q}$, can converge to a degenerate hyperplane of the form $V_{\mathcal{L}}$. Consider a sequence $Q_{(x_n,y_n)}$, for $(x_n,y_n)$  in $\R^{p-1,q}$. If $(x_n,y_n)$ converges to $(x_\infty,y_\infty)$, then clearly $Q_{(x_n,y_n)}$ converges to $Q_{(x_\infty,y_\infty)}$, which still gives an element of $\pd\HH^{p,q}$. So let us assume that $(x_n,y_n)$ diverges in $\R^{p-1,q}$.

Up to extracting a subsequence, we can assume that $x_n/\|x_n\|\to u$ and $y_n/\|y_n\|\to v$, for $(u,v)\in S^{p-2}\times S^{q-1}$.  Assume $(u,v)$ is a lightlike vector in $\R^{p-1,q}$ and moreover that $\|x_n\|-\|y_n\|$ has a finite limit $a\in\R$. Then dividing by $\|y_n\|$ the defining equation for $Q_{(x_n,y_n)}$, namely
$\|x-x_n\|^2-\|y-y_n\|^2+z^2=0$, and observing that  by our assumption $\|x_n\|/\|y_n\|\to 1$, we obtain that  $Q_{(x_n,y_n)}$ converges to the vertical hyperplane of equation
$x\cdot u-y\cdot v=a~,$
where $\cdot$ is the Euclidean inner product. Namely, the limit is a vertical hyperplane $V_{\mathcal{L}}$ whose underlying vector space is orthogonal to $(u,v)$. 

One can then show that otherwise $Q_{(x_n,y_n)}$ escapes from all compact sets of the half-space: in this situation the sequence $(x_n,y_n)$ in $\pd\HH^{p,q}$ converges to the point $\infty\in\pd_\infty\HH^{p,q}$.
\end{remark}

\subsection{Examples}

Let us now describe the topology of $\partial_\infty\HH^{p,q}$ in two definite examples.

\begin{example}
Let us first consider $\HH^{1,n}$, namely the half-space model of \emph{minus} the de Sitter space. In this case $\partial\HH^{1,n}$ is conformal to $\R^{0,n}$, hence is negative definite, therefore there are no degenerate affine hyperplanes in $\partial\HH^{1,n}$. In other words, $\partial_\infty\HH^{1,n}$ is the one-point compactification of $\partial\HH^{1,n}\cong\R^n$, and therefore is homeomorphic to the sphere $S^n$. This is not surprising indeed, as the $(n+1)$-dimensional de Sitter space shares the same boundary at infinity as the hyperbolic space $\Hyp^{n+1}$ of the same dimension. 

From the point of view of $\HH^{1,n}$, this corresponds to the fact that a sequence of degenerate totally geodesic hypersurfaces $Q_{y_n}$ of equation $\|y-y_n\|^2-z^2=0$, which is a cone over $(y_n,0)$ (see Figure \ref{fig:quadrics} on the right), escapes from compact sets in the half-space if the sequence $y_n$ is diverging in $\R^n$.
\end{example}

\begin{example}
Let us now consider a more interesting situation, namely the Anti-de Sitter half-space $\HH^{n,1}$. In this case $\pd_\infty\HH^{n,1}$ decomposes as the disjoint union of $\pd\HH^{n,1}$, which is a copy of the $n$-dimensional Minkowski space $\R^{n-1,1}$, the singleton $\{\infty\}$, and the space of vertical hyperplanes $V_{\mathcal{L}}$. The latter is in bijection with the space of degenerate hyperplanes in $\R^{n-1,1}$, which is a trivial bundle $S^{n-2}\times\R$, where the $S^{n-2}$ factor determines the orthogonal direction (\emph{i.e.} the projectivization of the cone $\{\|x\|^2-y^2=0\}$ of null directions), and the $\R$ factor the intercept on the $y$ axis. The complement $\pd_\infty\HH^{n,1}\setminus \pd\HH^{n,1}$ is therefore the one-point compactification of $S^{n-2}\times\R$.

When $n=2$, $S^0\times \R$ is the disjoint union of two lines, and its one-point compactification is homeomorphic to a wedge sum of two circles. Hence we directly recover the fact that $\partial_\infty\Hyp^{2,1}$ is homeomorphic to a torus $S^1\times S^1$. Indeed, from Remark \ref{rmk:convergence} we see that $\R^{1,1}$ is compactified by adding a point to compactify every line of the form $y=x+a$ (this is the first copy of $\R$ in $S^0\times\R$), and a point for every line of the form $y=-x+a$ (the second copy of $\R$). By adding the point $\infty$, we then see that the obtained topology is that of a torus, see Figure \ref{fig:torus}. Compare also (\cite[Appendix]{dan:phd}) for a more algebraic approach to this compactification.
\end{example}

\begin{figure}
\begin{center}
 \includegraphics[height=5cm]{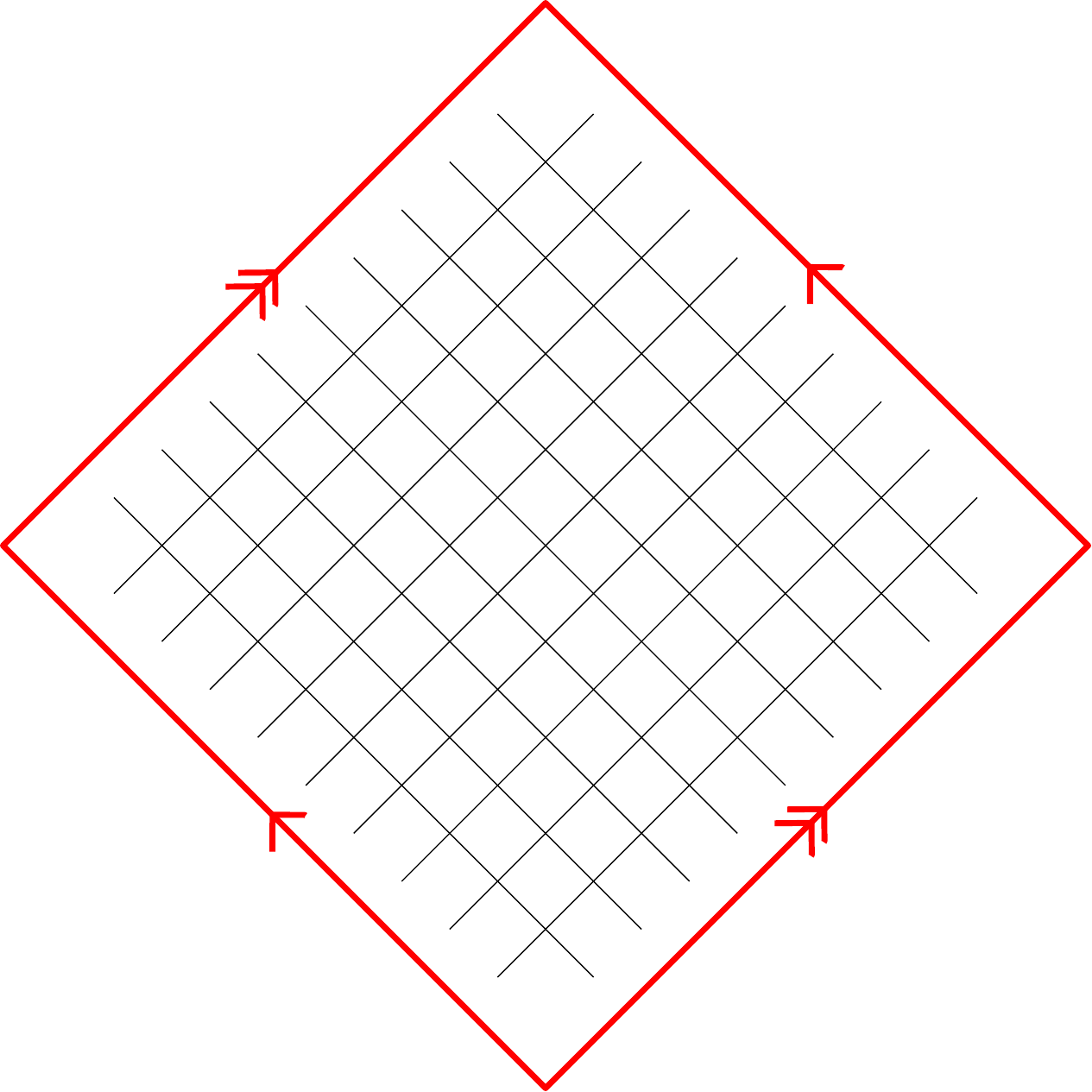}
\end{center}
\caption{The compactification of $\pd\HH^{2,1}$, which is a copy of $\R^{1,1}$ represented by the interior of the diamond, inside $\pd_\infty\HH^{2,1}$. The lines at $\pm 45^\circ$ represent the degenerate affine subspaces in $\R^{1,1}$, and each of them is compactified to a different point. The point $\infty$ then corresponds to the vertices of the diamond. The identifications of the sides clearly give the topology of a torus on $\pd_\infty\HH^{2,1}$.}\label{fig:torus}
\end{figure}

\subsection{Geodesics revisited}

To conclude this section, we discuss again the geodesics in $\HH^{p,q}$, now in terms of their endpoints in $\partial_\infty\HH^{p,q}$. Indeed, in $\Hyp^{p,q}$ the geodesics have the following topological behaviour:
\begin{itemize}
\item Spacelike geodesics converge to two different points in $\partial_\infty\HH^{p,q}$ at the two ends.
\item Lightlike geodesics converge to the same point in $\partial_\infty\HH^{p,q}$ at the two ends.
\item Timelike geodesics are closed, hence do not intersect $\partial_\infty\HH^{p,q}$.
\end{itemize}
We will classify geodesics, distinguishing their type as usual, in relation with their endpoints. 

\begin{remark}\label{rmk:convergence2}
Before stating the results, we give a preliminary observation that will be used repeatedly. It will be important to understand when a sequence of points $(x_n,y_n,z_n)\in \HH^{p,q}$ converges to a point of $\partial_\infty\HH^{p,q}$. In Section \ref{sec:boundary hyperplanes} we explained this for a sequence in $\partial\HH^{p,q}$, \emph{i.e.} for $z_n\equiv 0$ (see the discussion preceding Definition \ref{defi:topology bdy}).

When the points are in the interior, we can apply a similar consideration, namely the fact that a sequence of points $X_n\in\Hyp^{p,q}$ converges (projectively) to $X\in \pd_\infty\Hyp^{p,q}$ if and only if the lightcone emanating from $X_n$ converges to the totally geodesic degenerate hypersurface which corresponds to the orthogonal complement of $X$. Hence to check that a sequence $(x_n,y_n,z_n)\in \HH^{p,q}$ converges to a point of $\pd_\infty\HH^{p,q}$, which we recall is identified to the space of totally geodesic degenerate hypersurfaces, it suffices to check the convergence of the lightcones from $(x_n,y_n,z_n)\in \HH^{p,q}$ (as in Figure \ref{fig:lightcone}).

In particular, it is clear that the topology on $\HH^{p,q}\cup\partial\HH^{p,q}$ coincides with that of the closed half-space, because if $(x_n,y_n,z_n)\to (x_0,y_0,0)$, then the lightcones $\|x-x_n\|^2-\|y-y_n\|^2+|z-z_n|^2=0$ converge to the totally geodesic degenerate hypersurface $\|x-x_0\|^2-\|y-y_0\|^2+|z|^2=0$.
\end{remark}

Let us first consider spacelike geodesics, beginning with the case where the two endpoints are both in $\partial\HH^{p,q}\subset \partial_\infty\HH^{p,q}$.

\begin{proposition}\label{prop:spacegeo2}
Let $(x_0,y_0),(x_0',y_0')\in \partial\HH^{p,q}$ and let $(u,v)=(x_0'-x_0,y_0'-y_0)$ and $(x_m,y_m)=((x_0+x_0')/2,(y_0+y_0')/2)$. Then:
\begin{itemize}
\item If $\|u\|>\|v\|$, then the unique  geodesic of $\Hyp^{p,q}$ with endpoints $(x_0,y_0)$ and $(x_0',y_0')$ is contained in $\HH^{p,q}$, and is the ellipse of eccentricity $e_S(u,v)$ with center $(x_m,y_m)$.
\item If $\|u\|<\|v\|$, then the unique  geodesic of $\Hyp^{p,q}$ with endpoints $(x_0,y_0)$ and $(x_0',y_0')$ is contained in $\HH^{p,q}$ except for one point, and its intersection with $\HH^{p,q}$ consists of the two upper half-branches of the hyperbola of eccentricity $e_S(u,v)$ with center $(x_m,y_m)$.
\item If $\|u\|=\|v\|$, there is no geodesic with endpoints $(x_0,y_0)$ and $(x_0',y_0')$.
\end{itemize}
\end{proposition}
Recall that the value $e_S(u,v)$ of the eccentricity appears in Proposition \ref{prop:spacelike geo1}.
\begin{proof}
There is not much left to prove here. The first point follows from Proposition \ref{prop:spacelike geo1}. For the third point, it is known that if two points in $\partial_\infty\Hyp^{p,q}$ are connected by a lightlike segment in the boundary, then they are not connected by a spacelike geodesic; however the non-existence also follows from Proposition \ref{prop:spacelike geo1}. For the second point, using again Proposition \ref{prop:spacelike geo1}, the only thing left to prove is that the two half-branches of the same hyperbola are parts of the same spacelike geodesic in $\Hyp^{p,q}$, and are separated by a single point. (We have showed that these branches are incomplete on the upper end, so they certainly converge to the interior of $\Hyp^{p,q}$, since $\Hyp^{p,q}$ is geodesically complete.)

To prove this statement, we can apply the isometry group of $\HH^{p,q}$ and reduce to the curve parameterized by $y_1(t)=\pm\cosh(t)$, $z(t)=\sinh(t)$, and all the other coordinates identically zero (exactly as we did in the proof of Proposition \ref{prop:spacelike geo1}). A direct computation shows that $\iota_{p,q}$ maps this curve to a curve in $\Hyp^{p,q}$ whose only nonzero coordinates are $X_p=1/\sinh(t)$ and $X_{p+1}=\pm\cosh(t)/\sinh(t)$. Clearly these points all lie on the same geodesic, because they are contained in a unique 2-plane in $\R^{p,q+1}$, and the limit as $t\to +\infty$ is the same point in $\Hyp^{p,q}$ (\emph{i.e.} after taking the quotient by $\{\pm\mathrm{Id}\}$) regardless of the sign $\pm$ in front of $y_1(t)$. This concludes the proof.
\end{proof}

\begin{remark}\label{rmk:timegeo2}
A very similar computation shows that timelike geodesics of $\HH^{p,q}$ are mapped to the complement of a point on a (closed) timelike geodesic of $\Hyp^{p,q}$. Indeed, up to isometry, we can reduce to the branch of hyperbola given by   $	y_{1}(t)=\sinh(t)$, $z(t)=\cosh(t)$ and all the other coordinates identically zero. One can then show that the limit of the image under the embedding $\iota_{p,q}$  is the same point in $\Hyp^{p,q}$ as $t\to \pm\infty$.
\end{remark}

The case where one point is on $\partial\HH^{p,q}$ and the other is $\infty$ is very easy to deal with.

\begin{proposition}
Let $(x_0,y_0)\in \partial\HH^{p,q}$. The unique geodesic with endpoints $(x_0,y_0)$ and $\infty$ is the vertical line over $(x_0,y_0)$.
\end{proposition}
\begin{proof}
Applying Remark \ref{rmk:convergence2}, the endpoints of the vertical line over $(x_0,y_0)$ are clearly $(x_0,y_0)$ and $\infty$, for the lightcone over $(x_0,y_0,z)$ converges to the totally geodesic hypersurface $\|x-x_0\|^2-\|y-y_0\|^2+|z|^2=0$  as $z\to 0$, and escapes from compact sets as $z\to+\infty$.
\end{proof}

We are only left with the case where one point is on $\partial\HH^{p,q}$, and the other is represented by a totally geodesic hypersurface $V_{\mathcal L}$. Indeed, after proving the next proposition, and comparing with Proposition \ref{prop:spacelike geo1}, we see that \emph{a posteriori} there are no geodesics in $\HH^{p,q}$ connecting two points of the form $V_{\mathcal L}$ or $\infty$.

\begin{proposition}\label{prop:endpoint parabola}
Let $(x_0,y_0)\in \partial\HH^{p,q}$, $(u,v)\in\R^{p-1,q}$ such that $\|u\|=\|v\|$ and $\|u\|^2+\|v\|^2=1$, $a\in\R$ and $\mathcal L_{(u,v)}^a$ the degenerate affine hyperplane in $\partial\HH^{p,q}$ of equation
$$(x-x_0)\cdot u-(y-y_0)\cdot v=a~.$$ 
Then the unique geodesic with endpoints $(x_0,y_0)$ and $V_{\mathcal L_{(u,v)}^a}$ is the parabola
\begin{equation}\label{eq:parabola parameterized}
x(t)=x_0+\frac{t^2}{2a}u,\qquad y(t)=y_0+\frac{t^2}{2a}v, \qquad z(t)=t~.
\end{equation}
\end{proposition}
\begin{proof}
Up to a horizontal translation, which does not affect the conclusion of the statement, we can assume $(x_0,y_0)=(0,0)$. Set $\alpha=1/4a$. The lightcones over $(x(t),y(t),z(t))$ satisfy the equation $$\|x-2\alpha t^2u\|^2-\|y-2\alpha t^2v\|^2+|z-t|^2=0~.$$ Dividing by $t^2$ and using $\|u\|^2=\|v\|^2$, we see that these lightcones converge to the vertical hyperplane of equation $x\cdot u-y\cdot v=1/4\alpha=a$. Clearly $(x(t),y(t),z(t))$ converges to $(x_0,y_0,0)\in\pd\HH^{p,q}$ as $t\to 0$. This concludes the proof.
\end{proof}

\begin{remark}\label{rmk:endpoint parabola}
One might wonder what is the geometric interpretation of the parameter $a$, which encodes the relation between the parabola and its endpoint at infinity, seen as a vertical hyperplane that does \emph{not} contain the parabola itself. Let us describe the geometric intuition behind this relation. Given a parabola as in Equation \eqref{eq:parabola parameterized}, contained in a degenerate 2-plane $V_{\ell}$, where $\ell$ is an affine line directed by  $(u,v)$, one can uniquely express this parabola as the intersection of $V_\ell$ and a totally geodesic degenerate hypersurface, which is a lightcone over a point $(\hat x,\hat y,0)$. The vertical hyperplane  to which the parabola is asymtoptic to is then the unique vertical hyperplane $V_{\mathcal L}$ such that the underlying vector space of $\mathcal L$ is the orthogonal of $(u,v)$ and contains $(\hat x,\hat y,0)$. 
\end{remark}

Let us now move on to lightlike geodesics. 

\begin{proposition}
Given $(x_0,y_0)\in \partial\HH^{p,q}$, the lightlike geodesics of $\Hyp^{p,q}$ with endpoint $(x_0,y_0)$ are contained in $\HH^{p,q}$ except for one point, and their intersection with $\HH^{p,q}$ consists of two straight half-lines contained in the same vertical 2-plane. 
\end{proposition}
\begin{proof}
Up to a horizontal translation, it suffices to show that the half-lines $t\mapsto t(u,v,w)$ and $t\mapsto t(-u,-v,w)$, composed with the embedding $\iota_{p,q}$, converge to the same point in $\Hyp^{p,q}$ at $t\to +\infty$, which can be checked similarly to  Proposition \ref{prop:spacegeo2} and Remark \ref{rmk:timegeo2}.
\end{proof}

We now conclude our analysis by the only case left.

\begin{proposition}\label{lightgeo}
Given a degenerate affine hyperplane $\mathcal L$  in $\partial\HH^{p,q}$, the lightlike geodesics of $\Hyp^{p,q}$ with endpoint $V_{\mathcal L}$ are the horizontal straight lines contained in the vertical hyperplane $V_{\mathcal L}$ itself.
\end{proposition}
\begin{proof}
By Remark \ref{rmk:convergence2}, one has to check that the lightcones emanating from $(x_0+tu,y_0+tv,z_0)$ converge to the vertical hyperplane through $(x_0,y_0)$ whose underlying vector space is the orthogonal of $(u,v)$.  The computation is done exactly as in Remark \ref{rmk:convergence}.
\end{proof}

\section{Horospheres}\label{sec:horo}

We now briefly turn the attention to the study of the horospheres, in the half-space model. Let us recall the definition of horosphere in $\Hyp^{p,q}$.

\begin{definition}
	An horosphere in $\Hyp^{p,q}$ is a smooth hypersurface $S_a$ which is obtained as the projection in $\Hyp^{p,q}$ of
	\begin{equation}\label{eq:horosphere}
	\widetilde S_a=\{X\in\widetilde \Hyp^{p,q}\,|\,\langle X,V\rangle=a\}~,
	\end{equation}
	for some null vector $V\in \R^{p,q+1}$ (\emph{i.e.} $\langle V,V\rangle=0$) and some constant $a\neq 0$. We say that the horosphere $\widetilde S_a$ has point at infinity $[V]\in\partial_\infty\Hyp^{p,q}$.
\end{definition}

Observe that, when $V=\pd/\pd X_p-\pd/\pd X_{p+q+1}$, the corresponding horosphere $S_a$ is precisely the image of ${z=|a|}\subset \HH^{p,q}$ by the embedding $\iota_{p,q}$. We will use again this observation in the proof of Theorem \ref{prop:horo} below.	
	
For the sake of completeness, we provide a well-known characterization of horospheres in $\Hyp^{p,q}$, that generalizes a classical description in hyperbolic space.	
	\begin{lemma}
		The horospheres $S_{a}$ with point at infinity $[V]$ are precisely the smooth hypersurfaces orthogonal to all the spacelike geodesics having $[V]$ as an endpoint at infinity.
	\end{lemma}
	\begin{proof}		
		To check the statement, since orthogonality can be computed locally, we will work in the double cover $\widetilde{\mathbb{H}}^{p,q}$. Let $\widetilde{S}_{a}=\{X\in\widetilde{\mathbb{H}}^{p,q}\,|\,\langle X,V\rangle=a\}$ and $X\in\widetilde{S}_{a}$. Then 
	\begin{equation}\label{eq:tangent horo}
	T_{X}\widetilde S_{a}=T_{X}\widetilde{\mathbb{H}}^{p,q}\cap T_X(V^{\perp}+X)=X^{\perp}\cap V^{\perp}.
	\end{equation}
		
	Now, every geodesic of $\widetilde{\mathbb{H}}^{p,q}$ is contained in the intersection of $\widetilde{\mathbb{H}}^{p,q}$ with a linear 2-dimensional subspace. In particular, the unique spacelike geodesic $\gamma$ such that $\gamma(0)=X$ and having $[V]$ as an endpoint at infinity is contained in $\mathrm{Span}(X,V)$. So $\gamma'(0)\in \mathrm{Span}(X,V)$. Comparing with \eqref{eq:tangent horo}, we showed  that $\gamma'(0)$ intersects $\widetilde{S}_{a}$ orthogonally. 
		
	Finally, observe that every spacelike geodesic in $\Hyp^{p,q}$ with endpoint at infinity $[V]$ intersects $S_a$. Indeed, working again in the double cover, the preimages of spacelike geodesics of $\Hyp^{p,q}$ are the intersection of $\widetilde\Hyp^{p,q}$ with linear 2-dimensional subspaces. Given such a subspace containing the vector $V$, pick $X$ such that  $\langle X,V\rangle=a$.  Then for every $\lambda\in\R$ we have $\langle X-\lambda V,V\rangle=a$. Choosing $\lambda=(\langle X,X\rangle+1)/2a$, we obtain
	$$\langle X-\lambda V,X-\lambda V\rangle=\langle X,X\rangle-2\lambda a=-1,$$
	hence  $X-\lambda V\in\widetilde{\mathbb{H}}^{p,q}$, and therefore $X-\lambda V\in\widetilde S_a$. This concludes the proof.
	\end{proof}

	Despite the term horospheres, which is borrowed from classical hyperbolic geometry, horospheres are not topologically spheres. The boundary at infinity $\partial_\infty S_a$ of a horosphere $S_a$, namely its frontier in $\partial_\infty\Hyp^{p,q}$, is precisely the lightcone in $\partial_\infty\Hyp^{p,q}$ from $[V]$; then $S_a\cup \partial_\infty S_a$ is homeomorphic to $\partial_\infty\Hyp^{p,q}$.

In this section we will describe the horospheres in the half-space model $\HH^{p,q}$.

\begin{theorem}\label{prop:horo}
	The horospheres of $\HH^{p,q}$ are, for a parameter $c>0$:
	\begin{enumerate}
		\item horizontals hyperplanes $\{z=c\}$, if the point at infinity is $\infty$;
		\item wedges of hyperplanes of the form 
		$$z=c|x\cdot u-y\cdot v+d|$$
		if the point at infinity corresponds to the vertical hyperplane $V_\mathcal L$, for $\mathcal L$ the hyperplane of equation
		$x\cdot u-y\cdot v+d=0$ (for $(u,v)\in\R^{p-1,q}$ a null vector and $d\in\R$);
		\item piecewise quadric hypersurfaces of the form 
		$$\|x-x_0\|^2-\|y-y_0\|^{2}+\left(z\pm c\right)^2=c^2$$
		if the point at infinity is $(x_0,y_0,0)\in\pd\HH^{p,q}$.	
	\end{enumerate}
\end{theorem}

See Figures \ref{fig:horo12} and \ref{fig:horo3}.

\begin{figure}[htb]
\centering
\begin{minipage}[c]{.5\textwidth}
\centering
\includegraphics[height=3.8cm]{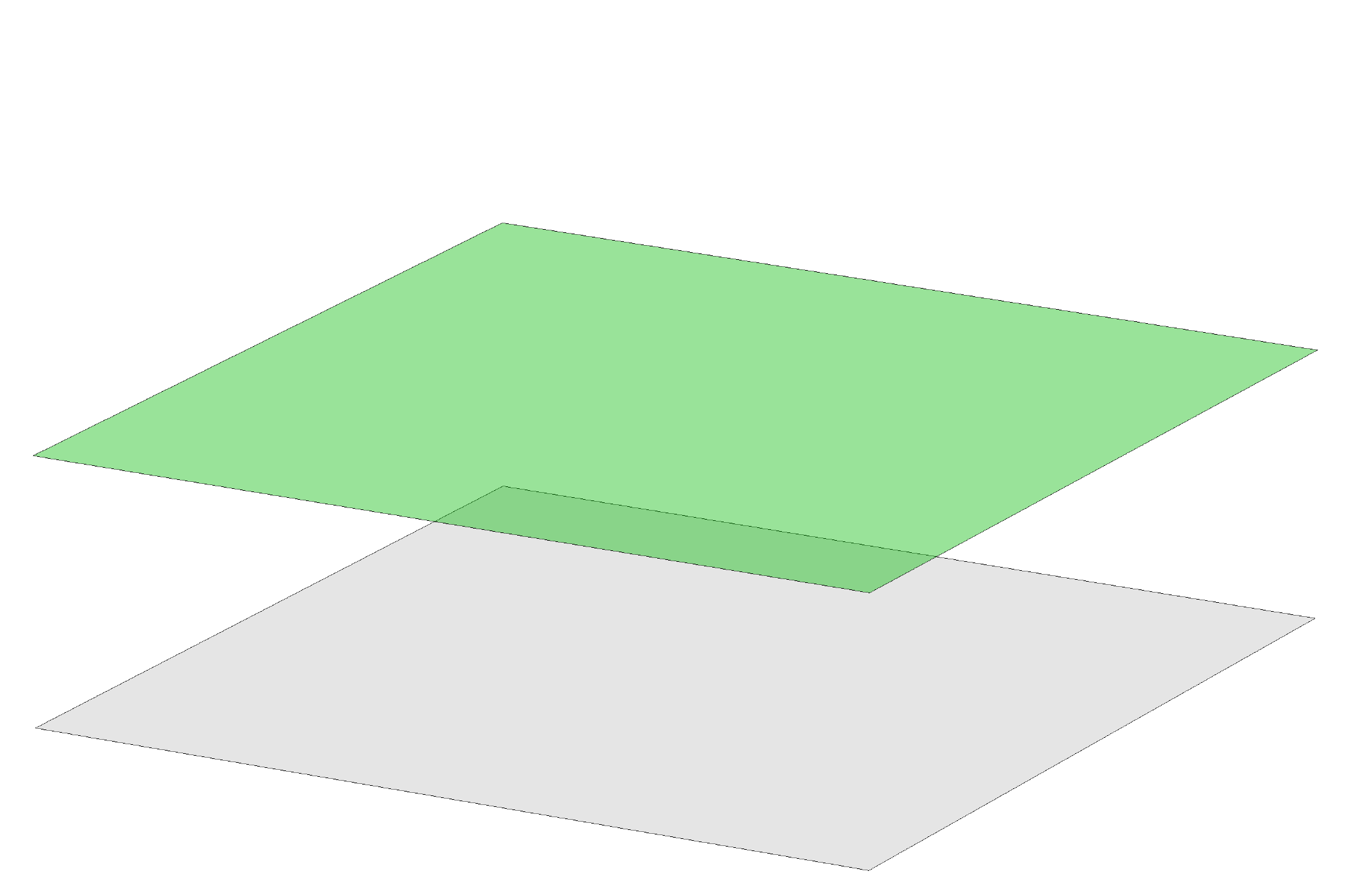}
\end{minipage}%
\begin{minipage}[c]{.5\textwidth}
\centering
\includegraphics[height=3.6cm]{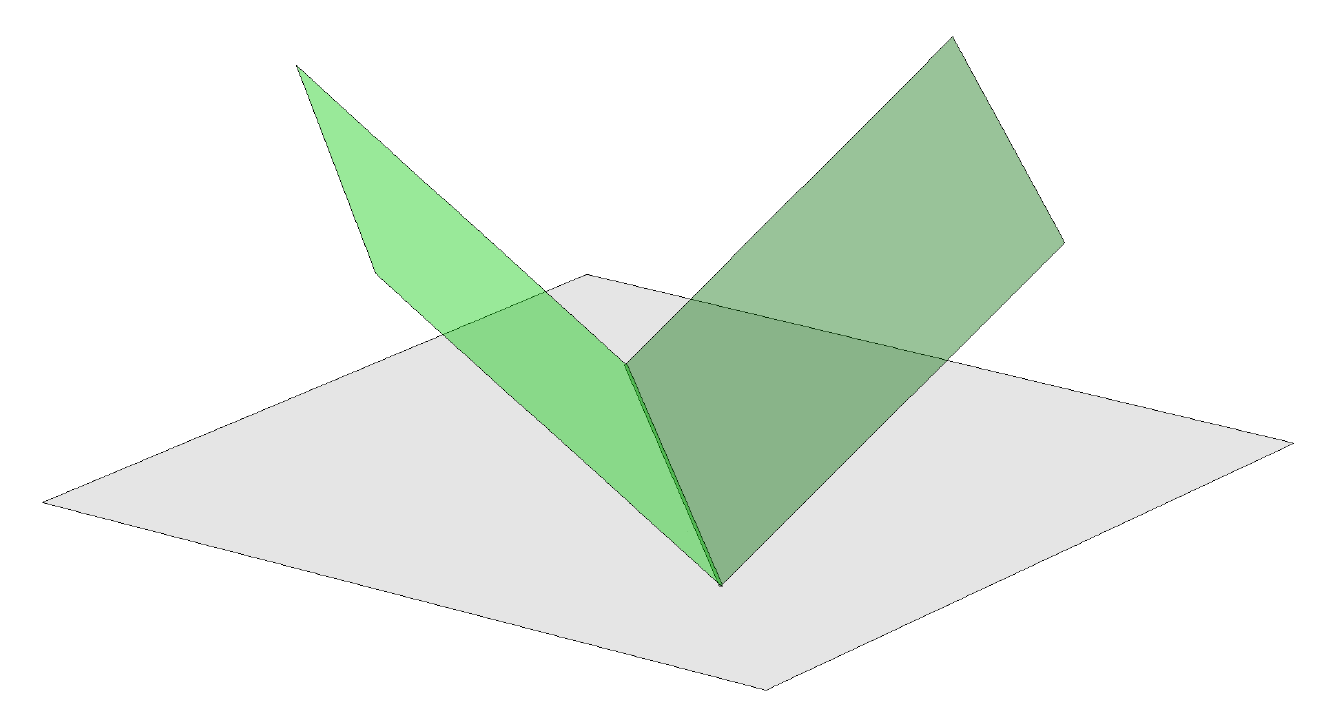}
\end{minipage}%
\caption{Horizontal horospheres, and wedges of hyperplanes.} \label{fig:horo12}
\end{figure}

\begin{figure}
\begin{center}
 \includegraphics[height=6cm]{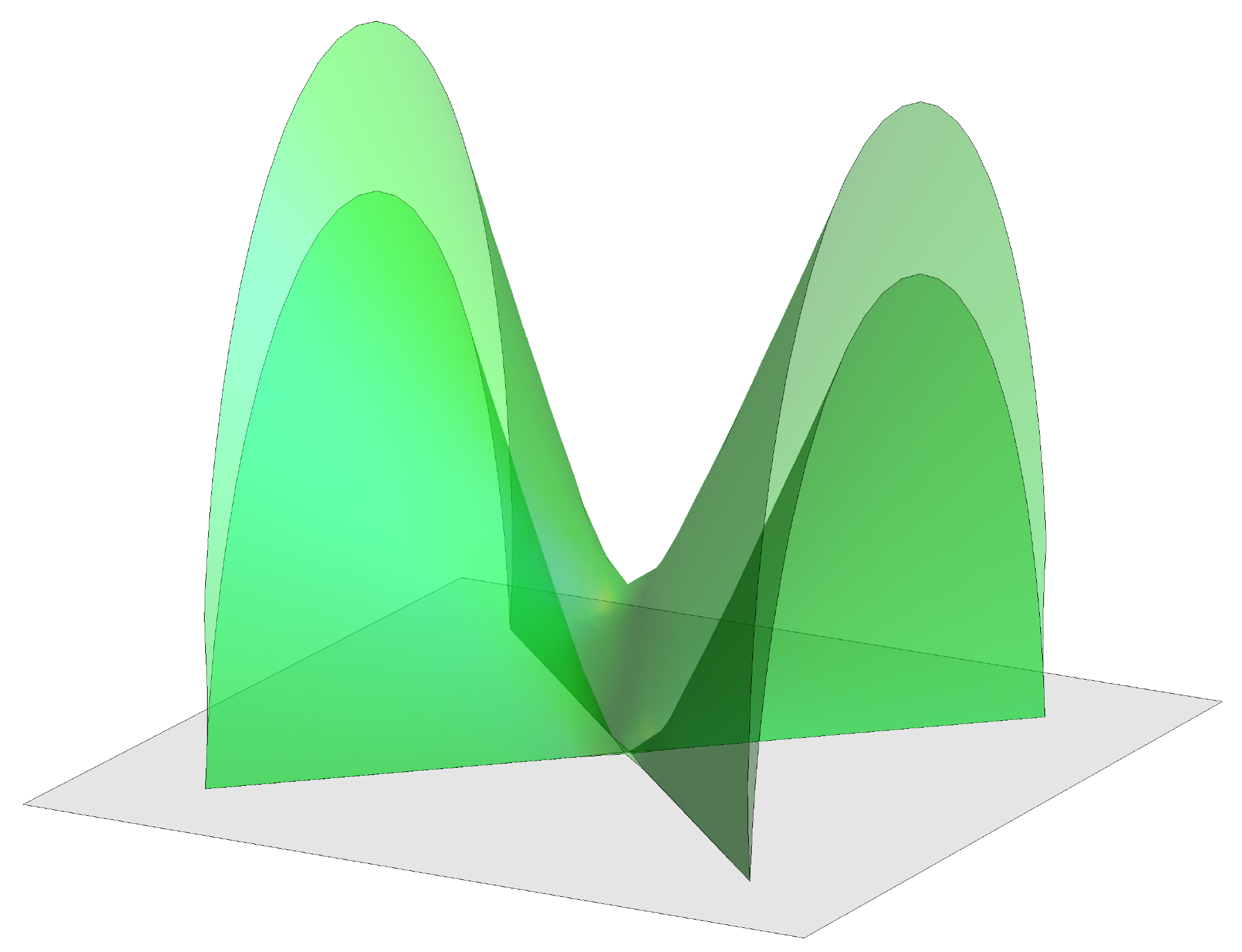}
\end{center}
\caption{Horospheres in $\HH^{2,1}$ corresponding to a point in $\partial\HH^{2,1}$.}\label{fig:horo3}
\end{figure}

\begin{proof}	
	Recall that we introduced the embedding $\tilde\iota_{p,q}\colon\HH^{p,q}\to \widetilde\Hyp^{p,q}$ in the proof Proposition \ref{emb}, that induces the embedding $\iota_{p,q}\colon\HH^{p,q}\to \widetilde\Hyp^{p,q}$ in the quotient. Also observe that every point in a horosphere $S_a$ has two preimages $X$ in $\widetilde\Hyp^{p,q}$, which satisfy either $\langle X,V\rangle=a$ or $\langle X,V\rangle=-a$. Hence to determine the horospheres in $\HH^{p,q}$ (or more precisely, the portion of horospheres contained in $\HH^{p,q}$) it suffices to find the preimage of $|\langle X,V\rangle|=a$ under $\tilde\iota_{p,q}$, for $a>0$.
	
	In the first case, the point at infinity is $\infty$ and corresponds to $[V]$ in $\partial_\infty\Hyp^{p,q}$, where $V$ is defined by $V_p=1$, $V_{p+q+1}=-1$, and all the other $V_i$'s vanish. In Proposition \ref{emb} we showed that, via $\tilde\iota_{p,q}$, $X_{p}+X_{p+q+1}=1/z$. This shows that the level sets $\{z=c\}$ are precisely the preimages of $\langle X,(0,1,0,-1)\rangle=a$, where $a=1/c$.
	
	Consider now the case where the point at infinity is $V_\mathcal L$. Up to a translation, we can assume that $d=0$, namely $\mathcal L$ is a degenerate hyperplane in $\R^{p-1,q}$ containing the origin. The corrisponding point $[V]\in\partial_\infty\Hyp^{p,q}$ is $[u:0:v:0]$. Hence the preimage of the horosphere defined by the equation $|\langle X,V\rangle|=a$ is the set of solutions of 
	$$|x\cdot u-y\cdot v|=az$$
	which concludes the second case, again by setting $c=1/a$.

	Finally, we consider the case where the point at infinity is $(x_0,y_0,0)\in\partial\HH^{p,q}$. Up to translation, we can assume $(x_0,y_0)=(0,0)$. The corresponding point in $\partial_\infty\Hyp^{p,q}$ is $[V]$, where $V_p=V_{p+q+1}=1$ and the other coordinates of $V$ are zero. Hence we need to determine the preimage of those $X$ satisfying $$|\langle X,(0,1,0,1)\rangle|=|X_p-X_{p+q+1}|=a.$$ Observe that  $X_{p}-X_{p+q+1}=-(h(x,y)+z^2)/z$, hence we obtain the equation $|h(x,y)+z^2|=az$. A simple manipulation gives the equivalent expression
	$$\|x\|^{2}-\|y\|^{2}+\left(z\pm\frac{a}{2}\right)^2=\frac{a^2}{4}~,$$
	which is the desired formula, for $c=a/2$.
\end{proof}

\begin{remark}
We conclude by remarking that the proof of Theorem \ref{prop:horo} could have been done by checking directly that the hypersurfaces of the three types are orthogonal to all the spacelike geodesics in $\HH^{p,q}$ which share an endpoint in $\partial_\infty\HH^{p,q}$. This is evident for the horizontal horospheres $\{z=c\}$, which are orthogonal to all vertical geodesics, \emph{i.e.} with endpoint $\infty$. 

For the wedges of hyperplanes, one can show directly, using Proposition \ref{prop:endpoint parabola}
 and Remark \ref{rmk:endpoint parabola}, that the union of the hyperplanes $z=\pm(x\cdot u-y\cdot v)$ is orthogonal to all the parabolas whose endpoint corresponds to the vertical hyperplane $V_\mathcal L$, $\mathcal L=\{x\cdot u-y\cdot v=0\}$, namely those parabolas which are obtained as the intersection of a vertical 2-plane projecting to an affine line directed by $(u,v)$, and a lightcone based on a point of $\mathcal L$. 

Finally, for the horospheres of the third type, one could check that these are orthogonal to the geodesics with endpoint in $(x_0,y_0,0)$ in the following way. Up to an isometry of $\HH^{p,q}$, assume $(x_0,y_0)=(0,0)$ and $c=1$. Then  one ``sweeps'' the hypersurface by curves of four types. The first case is that of a curve contained in a vertical 2-plane which is positive definite. Up to an isometry of the form $(x,y,z)\mapsto (A(x,y),z)$, which leaves the horosphere invariant, it suffices to consider the curve $x_1(t)=\sin(t)$, $z(t)=\cos(t)+1$, and all the other coordinates identically zero. Then one shows that this curve is orthogonal to all the geodesics with endpoint $(0,0)$ that it intersects, which are ellipses (circles, in this specific situation). This is exactly analogous to the half-space model of $\Hyp^n$. Second, one consider curves in a vertical 2-plane which is indefinite. Again up to isometry, one reduces to two curves, defined by $y_1(t)=\sinh(t)$, and $z(t)=\cosh(t)+1$ or $z(t)=\cosh(t)-1$. These are orthogonal to all the spacelike geodesics which are hyperbolas and have $(0,0)$ as an endpoint. In the former case, the curve intersects the branch containing $(0,0)$; in the latter, the other branch. Finally, the horizontal planar curves contained in the horosphere are trivially ortogonal to all parabolas with endpoint $(0,0)$, because they are lightlike and contained in a degenerate vertical 2-plane.
\end{remark}

\section{Isometries}\label{sec:isometry}

Let us conclude this paper by describing the isometries of $\HH^{p,q}$, and the action of those of $\Hyp^{p,q}$, in the half-space model.

\subsection{The isometry group $\isom(\HH^{p,q})$}

We remarked in Section \ref{sym} that $\iota_{p,q}$ induces a monomorphism $G\to\isom(\Hyp^{p,q})$, as a consequence of the fact that local isometries between open neighbourhoods of $\Hyp^{p,q}$ uniquely extend to global isometries. From our study of geodesics in Section \ref{sec:geo}, we can deduce that the group $G$ introduced in Lemma \ref{lemma:G} is the full isometry group of the half-space model. 

\begin{theorem}\label{cor:G}
When $q\geq 1$, the group $G$ coincides with the isometry group $\isom(\HH^{p,q})$. Moreover, $G$ corresponds precisely to the isometries of $\isom(\Hyp^{p,q})$ that preserve the totally geodesic degenerate hyperplane $\Hyp^{p,q}\setminus\iota_{p,q}(\HH^{p,q})$.
\end{theorem}
\begin{proof}
Since $G<\isom(\HH^{p,q})$ acts transitively by Lemma \ref{lemma:G}, it suffices to prove that $\stab_{G}(0,0,1)=\stab_{\isom(\HH^{p,q})}(0,0,1)$.
We observe that  $\stab_{G}(0,0,1)$ is the subgroup of $\stab_{\isom(\HH^{p,q})}(0,0,1)$ preserving oriented vertical lines, \emph{i.e.} it consists of those isometries  $f$ such that $df_{(0,0,1)}(\pd/\pd z)=\pd/\pd z$. We claim that all isometries $f$ in $\stab_{\isom(\HH^{p,q})}(0,0,1)$ have this property. 

By contradiction, assume $df_{(0,0,1)}(\pd/\pd z)\neq \pd/\pd z$. First, if $df_{(0,0,1)}(\pd/\pd z)=-\pd/\pd z$, then a lightlike geodesic starting at $(0,0,1)$ and parameterized in such a way that the $z$-coordinate is increasing along the geodesic (hence incomplete) would be sent to another lightlike geodesic parameterized in such a way that the $z$-coordinate is decreasing (hence complete) which is an absurd since isometries preserve completeness of geodesics. Otherwise, $df_{(0,0,1)}(\pd/\pd z)^{\perp}\ne(\pd/\pd z)^{\perp}$. The horizontal hyperplane $(\pd/\pd z)^{\perp}\cong\R^{p-1,q}$ is generated by lightlike vectors, hence there exists a horizontal lightlike vector $v$ such that $df_{(0,0,1)}(v)$ is not horizontal. This is an absurd as lightlike geodesics are complete (in both directions) if and only if the initial velocity is horizontal (Lemma \ref{lightlike}), and again isometries preserve completeness.

The second part of the statement is clear, because every isometry of $\HH^{p,q}$ extends to an isometry of $\Hyp^{p,q}$ which preserves the image of $\iota_{p,q}$, hence also its complement. Conversely, every isometry of $\Hyp^{p,q}$ that preserves the image of $\iota_{p,q}$ induces an isometry of $\HH^{p,q}$, and therefore is in $G$.
\end{proof}

\subsection{Inversions}
In order to describe the action of the isometry group $\isom(\Hyp^{p,q})$ on the half-space model, we now introduce a new type of isometries, that are the analogous of inversions in hyperbolic geometry. Recall that, given a point $(x_0,y_0)\in\pd\HH^{p,q}$, $Q_{(x_0,y_0)}$ denotes the totally geodesic hypersurface made of lightlike geodesics with endpoint $(x_0,y_0)$, as in \eqref{eq:conex0y0}.

\begin{proposition}\label{prop:Jextends}
	The involution $\mathcal J\colon\HH^{p,q}\setminus Q_{(0,0)}\to\HH^{p,q}\setminus Q_{(0,0)}$ defined by	
	
$$
			(x,y,z)\mapsto\left(\mu(x,y,z) x,\mu(x,y,z) y,|\mu(x,y,z)|z\right)~,
$$
	where $\mu(x,y,z):=(\|x\|^2-\|y\|^2+z^2)^{-1}$, is an isometry which extends to a global isometry of $\Hyp^{p,q}$ via $\iota_{p,q}$. 
\end{proposition}		
	We remark that if $q=0$, $Q_{(0,0)}=\emptyset$ and $\mu>0$, hence we recover the fact that $\mathcal J$ is a global isometry of the hyperbolic space.

\begin{proof}
	Let $f\in \isom(\Hyp^{p,q})$ be the isometry induced by the reflection in the hyperplane $X_p=0$. To prove the statement, we show that the following diagram commutes:
	\[
	\begin{tikzcd}
		\HH^{p,q}\setminus Q_{(0,0)}\arrow[r,"\iota_{p,q}"]\arrow[d,"\mathcal J"] & \Hyp^{p,q}\arrow[d,"f"]\\
		\HH^{p,q}\setminus Q_{(0,0)}\arrow[r,"\iota_{p,q}"] & \Hyp^{p,q}	
	\end{tikzcd}.
	\]
	We first remark that 
	\begin{equation}\label{eq:involution}
	\mu(\mathcal J(x,y,z))=\mu(x,y,z)^{-2}(\|x\|^2-\|y\|^2+z^2)^{-1}=\mu(x,y,z)^{-1}~,
	\end{equation}
	which also immediately implies that $\mathcal J$ is an involution.  
	Observe that $\mathcal J$ is defined precisely on the complement of $\{\mu=0\}$. Suppose first $\mu>0$. Denote 
	$\tilde\iota_{p,q}(x,y,z)=(X_1,\ldots,X_{p+q+1})$ (these are defined in the proof of Proposition \ref{emb}) and  $\tilde\iota_{p,q}\circ \mathcal J(x,y,z)=(Y_1,\ldots,Y_{p+q+1})$. We have:
	\begin{align*}
		&Y_i=\frac{\mu x}{\mu z}=\frac{x}{z}=X_i && i=1,\dots p-1\\
		&Y_p=\frac{1-\mu}{2\mu z}=-\frac{1-h(x,y)-z^2}{2z}=-X_p&&\\
		&Y_{j+p}=\frac{\mu y}{\mu z}=\frac{y}{z}=X_{j+p} && j=1,\dots q\\
		&Y_{p+q+1}=\frac{1+\mu}{2\mu z}=\frac{1+h(x,y)+z^2}{2z}=X_{p+q+1}&&
	\end{align*}
where in the second and fourth line we have used \eqref{eq:involution} together with the fact that, in the notation of the embedding, $\mu(x,y,z)^{-1}=h(x,y)+z^2$. This shows that $\tilde\iota_{p,q}\circ \mathcal J=\tilde f\circ\tilde\iota_{p,q}$, where $\tilde f\in \mathrm{O}(p,q+1)$ is the reflection fixing the hyperplane $X_p=0$.

One immediately checks that $\tilde\iota_{p,q}\circ\mathcal J=-\tilde f\circ\tilde\iota_{p,q}$ when $\mu<0$. Since $\tilde f$ and $-\tilde f$ induce the same isometry on $\Hyp^{p,q}$, the claim is proved.
\end{proof}

\begin{remark}
We saw that the involution $\mathcal J$ corresponds to a reflection in $\isom(\Hyp^{p,q})$. In fact its fixed point set is the totally geodesic hypersurface 
$$\|x\|^2-\|y\|^2+z^2=1~.$$
The other inversions, fixing the general totally geodesic hypersurface of the form \eqref{eq:quadric} for $c>0$, can be easily found conjugating $\mathcal J$ by elements of $G$.
\end{remark}

\subsection{Action of $\isom(\Hyp^{p,q})$}

We conclude by describing the action of the full isometry group $\isom(\Hyp^{p,q})$ on $\HH^{p.q}$. Roughly speaking, the subgroup $G$ and the inversion $\mathcal J$ (or more precisely, their extensions to $\Hyp^{p,q}$) generate $\isom(\Hyp^{p,q})$. 

\begin{theorem}
	Any isometry of $\Hyp^{p,q}$ can be written in $\HH^{p,q}$ as the composition of elements of $G$ and $\mathcal J$.
\end{theorem}
\begin{proof}
	Since $G$ corresponds precisely to the stabilizer of a point in $\pd_\infty\Hyp^{p,q}$ by Theorem \ref{cor:G}, it suffices to show that the elements of $G$, together with $\mathcal J$, induce a transitive action on $\pd_\infty\HH^{p,q}$. Clearly $G$ acts transitively on $\pd\HH^{p,q}$, while $\mathcal J$ maps $(0,0)\in\pd\HH^{p,q}$ to $\infty$, which is trivial since $\iota_{p,q}(0,0,0)=[0:1:0:1]$ and $\iota_{p,q}(\infty)=[0:-1:0:1]$. Also, $G$ acts transitively on the degenerate vertical hyperplanes of the form $V_{\mathcal L}$. Hence it remains to show that in the subgroup generated by $G$ and $\mathcal J$, there is an element that maps some point in $\pd\HH^{p,q}$ to some point in $\pd_\infty\HH^{p,q}$ that corresponds to a vertical hyperplane $V_{\mathcal L}$. But this clear because by Proposition \ref{prop:Jextends} $\mathcal J$ extends to an element $f\in\isom(\Hyp^{p,q})$, whose action on $\pd_{\infty}\Hyp^{p,q}$ is a homeomorphism, hence it maps a neighbourhood of $\infty$ (which contains elements of the form $V_{\mathcal L}$) to a neighborhood of $(0,0)$ (which only contains points in $\pd\HH^{p,q}$). 
	\end{proof}

%
%

\end{document}